\documentclass[12pt]{extarticle}
\usepackage[margin=0.3in, top=1in, bottom=1in]{geometry}
\usepackage[utf8]{inputenc}
\usepackage{amsmath,amssymb,amsthm,mathtools}
\usepackage{thmtools, thm-restate}
\usepackage{changepage}
\usepackage[T1]{fontenc}
\usepackage{eulervm,eucal,eufrak}
\usepackage{bbold}
\usepackage{concrete}
\usepackage{lipsum,hyperref,datetime2}
\usepackage[dvipsnames]{xcolor}
\hypersetup{
    colorlinks=true,
    filecolor=red,
    linkcolor=blue,
    citecolor=blue,
    urlcolor=blue,
}
\usepackage{quiver}
\usepackage{graphicx}
\graphicspath{ {./graphics/} }
\usepackage[backend=biber,style=alphabetic]{biblatex} 
\bibliography{jin.bib} 
\newtheoremstyle{mystyle}
  {}
  {}
  {}
  {}
  {\bfseries}
  {}
  { }
  {\thmname{#1}\thmnumber{ #2}\thmnote{ (#3)}}
\numberwithin{equation}{section} 
\theoremstyle{mystyle}
\newtheorem{theorem}[equation]{Theorem}
\AtEndEnvironment{theorem}{\null\hfill$\diamond$}%
\newtheorem{convention}[equation]{Convention}
\AtEndEnvironment{convention}{\null\hfill$\diamond$}%

\AtEndEnvironment{proposition}{\null\hfill$\diamond$}%
\newtheorem{definition}[equation]{Definition}
\AtEndEnvironment{definition}{\null\hfill$\diamond$}%
\newtheorem{lemma}[equation]{Lemma}
\AtEndEnvironment{lemma}{\null\hfill$\diamond$}%

\AtEndEnvironment{corollary}{\null\hfill$\diamond$}%
\newtheorem{example}[equation]{Example}
\AtEndEnvironment{example}{\null\hfill$\diamond$}%
\newtheorem{remark}[equation]{Remark}
\AtEndEnvironment{remark}{\null\hfill$\diamond$}%
\newtheorem{notation}[equation]{Notation}
\AtEndEnvironment{notation}{\null\hfill$\diamond$}%
\let\oldproof\proof
\definecolor{my-dark-gray}{gray}{0.13}
\newcommand{\m}{\mbox{-}} 
\renewcommand{\proof}{\color{my-dark-gray}\oldproof}

\newenvironment{Proof}[1][Proof]
  {\proof[#1]\leftskip=0.5cm\rightskip=0cm}
  {\endproof}

\newcommand{\Rep}{\operatorname{Rep}}
\newcommand{\Obj}{\operatorname{Obj}}

\newcommand{\Kar}{\operatorname{Kar}}
\newcommand{\Adm}{\operatorname{Adm}}
\newcommand{\Aut}{\operatorname{Aut}}
\newcommand{\End}{\operatorname{End}}

\newcommand{\id}{\operatorname{id}}
\newcommand{\cok}{\operatorname{cok}}
\newcommand{\im}{\operatorname{im}}

\title{Categorical Center of Higher Genera and 4D Factorization Homology}
\author{Jin-Cheng Guu}
\date{}

\begin{document}

\maketitle
\begin{flushright}
  Compiled Time: [\today\,\DTMcurrenttime] \qquad.
\end{flushright}

\abstract{Many quantum invariants of knots and $3$-manifolds
  (e.g. Jones polynomials) are special cases of the
  Witten-Reshetikhin-Turaev $3$D TQFT. The latter is in turn a
  part of a larger theory - the Crane-Yetter $4$D TQFT. In this
  work, we compute the Crane-Yetter theory for all (smooth and
  oriented) surfaces with at least one puncture. The results in
  general are constructed and called the categorical center of
  higher genera.}

\tableofcontents

\section*{Acknowledgement}

It is a great honor of the author to contribute to such a
beautiful theory. The author would like to express deep gratitude
to the author's advisor, Alexander Kirillov, for his guidance and
tremendous patience. Without him, the current work would not have
been possible. The author would also like to thank Ying Hong Tham
for several fruitful discussions.

\pagebreak

\section{Introduction}

Many quantum invariants of knots and $3$-manifolds (e.g. Jones
polynomials) are special cases of the Witten-Reshetikhin-Turaev
$3$D TQFT, which is in turn a part of a larger theory - the
Crane-Yetter $4$D TQFT (also called the CY model). The CY model,
first given as a state sum
\cite{crane-yetter-categorical-construction}, is believed to be a
fully-extended TQFT. In particular, its definition in
(co)dimension $2$ has been given in
\cite{fac-homo--kirillov-tham}.

In this work, we compute the Crane-Yetter theory for all (smooth
and oriented) surfaces with at least one puncture. To do that,
we first present a surface $\Sigma$ by a
combinatorial datum $\sigma$ (admissible gluing \ref{def/adm-gluing})
via a combinatorial construction ($\sigma$-construction
\ref{def/sigma-construction}) so that $\Sigma = \Sigma_{\sigma}$. We then
construct a category $Z_{\sigma}$ called the categorical center of
higher genera and show some of its basic properties
(section \label{section/algebraic-theory}). The main result
(\ref{main-statement}) proves an equivalence of finite semisimple
abelian categories $CY_{C}(\Sigma) \simeq Z_{\sigma}(C)$ for any premodular
category $C$ and any algebraically closed field $\mathbb{k}$ of
characteristic $0$.

This result generalizes almost all known results in (co)dimension
$2$ (section
\ref{section/previous-work-of-cy-in-codimension-two}). In
particular, it generalizes the Drinfeld categorical center to
higher genera.

\pagebreak

\section{Overview}

\subsection{Invariants as data representations}

On investigating a mathematical object, one starts with a
presentation, which is a description of the object from the
simpler ones. Though a presentation in principle fully describes
the object, it is yet an image distant to the essence of the
object. For example, a finite group presented by a finite set of
generators and relations is fully described but by no means well
understood. Similarly, though a space presented by a finite set
of (higher dimensional) triangles and gluing data is in principle
fully described, its topological structure can be obscured by the
complexity of the data. It is therefore natural to seek for
deeper understandings. One way often used is to represent the
presented data in different forms.

On representing data, two issues arise. First, how accurate is
the representation? Does it lose any information? And how much is
lost if it does? Second, the representation is again a
mathematical object, so it is also natural to ask how much we
understand the represented object. Is it easy to compute or
characterize? In many cases, one should balance between the two.
Indeed, a more accurate representation often teaches us less, an
extreme case being the presentation itself; a less accurate
representation to a larger degree simplifies the object and would
hopefully be more enlightening to the mortals, a simple example
being the size of a finite set: while it does not faithfully
represent the set as a mathematical object, it is useful to some
extent.

Such representations are called invariants. In the context of
topology, they are called topological invariants. The simplest
example is the notion of dimension, which assigns a given space
to a positive integer. The genus of a surface is another
classical instance among the integer-valued invariants. However,
an invariant needs not take numbers as its values. Numbers may be
easier to understand, but they forget too much information. Being
an topological invariant, the homology $H(-;\mathbb{Q})$ takes
vector spaces as values. Certainly, they do tell us more than the
dimension does. More sophisticated examples can take even more
complicated algebraic objects (groups, algebras, Hopf algebras..
etc) as values. For instance, a milestone in algebraic topology
is the celebrated theorem of Mandell.

\begin{theorem} \cite{mandell/cochains-and-homotopy-type}
  Finite type nilpotent spaces $X$ and $Y$ are weakly equivalent
  if and only if the $E^{\infty}$-algebras $C^{\star}(X)$ and $C^{\star}(Y)$
  are quasi-isomorphic.
\end{theorem}

Such leap of thought from numbers to higher objects is seminal.
Not only does it enable us to represent the objects more
accurately, perhaps more importantly, it allows us to represent
the relations among objects of interest by (higher) functors. \\

\subsection{TQFTs as higher invariants}

\begin{convention}[manifold, field]
  Throughout this paper, by a manifold of dimension $n$ we mean
  an oriented smooth manifold without boundary of real-dimension
  $n$; we also work over a fixed field $\mathbb{k}$ that is
  algebraically closed and with characteristic $0$.
\end{convention}

Another example of a invariant that takes vector spaces as
values, originally motivated by theoretical physics, is called a
topological quantum field theory (TQFT) à la Michael Atiyah
\cite{atiyah/tqft}. Roughly put, a TQFT, like a homology, assigns
vector spaces to some manifolds. It also, in a slightly different
manner, assigns in a compatible way a linear transformation to
the manifolds which are of one-dimension higher and bridge the
lower manifolds. Curiously, this gives a number-valued invariant
for closed manifolds: closed manifolds ``bridge'' the empty
manifolds, to which the trivial vector space is assigned, so the
linear transformations assigned to the closed manifolds are numbers.

This generalization from numbers to vector spaces does not stop
here. By viewing vector spaces as objects in a category, one can
bring this process to assigning objects in higher categories. An
instance of such is called an extended topological quantum field
theory (eTQFT) \cite{lurie/classify-eTQFT}.

\begin{example}[Witten-Reshetikhin-Turaev model]\label{example/wrt-model}
  Among manifolds, the one-dimensional and two-dimensional ones
  are fully classified. Though a presentation for the
  three-dimensional case is provided by knots and surgery theory
  (cf \cite[chap.12]{lickorish-knot-theory} and
  \cite{kirby1978}), it is based on the classification of knots,
  which is itself complicated and in fact still an ongoing
  research field. Therefore, invariants of $3$-manifolds are of
  interest.

  Around the $80$s, initiated by V Jones with his celebrated
  Jones polynomials, a new type of invariant was born in the
  context of mathematical-physics. For example, one can construct
  the invariants from von Neumann algebras, quantum groups, and
  rational conformal theories
  \cite[sec.1]{crane-yetter-categorical-construction}. It is
  natural to ask if they can be unified. The answer is positive
  in that all of them are special cases of the
  Witten-Reshetikhin-Turaev (WRT) model, whose input algebraic
  data is called a modular tensor category \cite{MTC-321tqft}).

  As an extended TQFT, WRT model assigns a number to each closed
  $3$-manifold, a vector space to each closed $2$-manifold, and a
  linear category to each closed $1$-manifold. However, it does
  not go on and assign a linear $2$-category in the $0$-th
  dimension. See \ref{remark/why-wrt-is-not-fully-extended} for
  more discussion.
\end{example}

\begin{example}[Turaev-Viro model]\label{example/tv-model}
  The Turaev-Viro (TV) model is first given as an invariant for
  $3$-manifolds in the form of a state sum
  \cite{turaev-viro-model}. It was clear then that the TV model
  gives a $3$-dimensional TQFT.

  Later, it is proven to be a fully extended TQFT in
  \cite{kirillov-balsam/turaev-viro-I}. Namely, it compatibly
  assigns a number to each closed $3$-manifold, a vector space to
  each closed $2$-manifold, a linear category to each closed
  $1$-manifold, and a linear $2$-category to each $0$-manifold
  (i.e. the point). It was also shown that
  $$TV_{C}(M) \simeq WRT_{Z(C)}(M)$$
  where $M$ is a closed manifold of dimension $2$ or $3$, $C$ is
  a spherical category and $Z(C)$ is its Drinfeld categorical
  center (\cite{balsam/turaev-viro-II} and
  \cite{balsam/turaev-viro-III}).
\end{example}

\begin{example}[Crane-Yetter model]\label{example/cy-model}
  The Crane-Yetter model is a $4$-dimensional analogue of the TV
  model. Its input algebraic data is a premodular category.
  It is the main topic of this paper, and will be treated
  thoroughly in its own section
  (\ref{section/cy-as-an-extended-tqft}).
\end{example}

\begin{theorem}[$\partial CY = WRT$]\label{thm/wrt-is-the-boundary-theory-of-cy}
  The WRT model is the boundary theory of CY, while the input
  algebraic data is a modular category. More precisely,
  for a modular category $C$ and a
  $4$-dimensional manifold $W$ possibly with boundary,
  $$CY_{C}(W) = \kappa^{\sigma(W)} WRT_{C}(\partial W),$$
  where $\sigma$ denotes the signature and $\kappa$ denotes a constant
  based on the input algebraic data (the modular category given
  in \ref{example/premodular-category--quantum-group}). This
  extends to manifolds with colored graphs \cite[Theorem
  2]{barrett/observables-in-tv-and-cy}.

  For a $3$-dimensional manifold $M$ possibly with boundary, the
  vector spaces $CY_{C}(M)$ (cf \ref{thm/cy-trivializes-over-mtc})
  and $WRT_{C}(\partial M)$ are widely believed to be equivalent. However,
  to our best knowledge a rigorous proof has not been provided.
\end{theorem}

\begin{remark}\label{remark/why-wrt-is-not-fully-extended}
  In \ref{example/wrt-model} we mentioned that the WRT model does
  not extend to the $0$-th dimension. There are two explanations
  for this fact.
  \begin{enumerate}
    \item From the Turaev-Viro model perspective
    \cite{balsam/turaev-viro-II}, for WRT to extend to a
    point, one needs the input modular category $C$ to be
    the Drinfeld center of a spherical fusion category $D$.
    This is not always the case.
    \item From the Crane-Yetter (CY) model perspective, the WRT
          model is a boundary theory of CY
          (\ref{thm/wrt-is-the-boundary-theory-of-cy}). However,
          a single point is not the boundary of any $1$-manifold:
          one needs at least two points.
  \end{enumerate}
  From the second point of view of \ref{example/wrt-model}, one
  sees that the WRT model, though successful and fruitful, is a
  part of larger theory - the Crane-Yetter model, which is the
  main TQFT we will focus on in this work.
\end{remark}

\subsection{$CY$ as an extended TQFT}\label{section/cy-as-an-extended-tqft}

In this paper, we focus on a specific extended TQFT, namely, the
Crane-Yetter theory. We provide some historical context, display
some known results, and address our main result.

\subsubsection{The Crane-Yetter theory $CY$}

Originated in \cite{crane-yetter-categorical-construction}, the
Crane-Yetter model was first defined as a state-sum. In
particular, for a triangulated $4$-manifold $M$ and a modular
tensor category $C$, one define
$$CY(M) = \Sigma_{c} D^{(n_{0} - n_{1})} \Pi_{\sigma} \dim c(\sigma) \Pi_{t} \dim c(t) \Pi_{\xi} \underline{15j}(c,\xi),$$
where $c$ runs through all ``colorings'', $n_{i}$ is the number
of simplices of dimension $i$ in the triangulation, $\sigma$ runs
through all triangles, $t$ runs through all tetrahedra, $\xi$ runs
through all $4$-simplices, and $\underline{15j}$ denotes the so
called $15j$-symbols. The upshot is that the sum is independent
of the triangularization, therefore defines an invariant of
$4$-manifolds. This holds even when the $C$ is a premodular category.

Later in \cite{crane-kauffman-yetter/crane-yetter-state-sum}, it
was known that the this seemingly complicated sum can be
expressed in terms of the Euler characteristic and the signature,
both being old and well-known topological invariants. While this
provides a combinatorial formula for the signature of $4$-folds,
it also means that the CY model with the input data being a
modular tensor category $C$ somehow trivializes. Recall that the
WRT model is the boundary theory of the CY model
(\ref{thm/wrt-is-the-boundary-theory-of-cy}). This provides a
hint for why the CY model trivializes - the input algebraic data,
namely the modular tensor categories, are too good for
$4$-manifolds. Therefore, other types of algebraic data should be
considered. For example, in this paper we focus on the premodular
categories.

On the other hand, as the TV model, the CY model is also a fully
extended TQFT. In particular, it gives a number to each closed
$4$-manifold, a vector space to each closed $3$-manifold, a
linear category to each closed $2$-manifold, a linear
$2$-category to each closed $1$-manifold (the circle), and
finally a linear $3$-category to each closed $0$-manifold (the
point). In this work, we will focus on dimension $2$, in which
the CY model takes linear categories as values.

\subsubsection{Previous work of $CY$ in (co)dimension two} \label{section/previous-work-of-cy-in-codimension-two}

Recent developments of the CY model include its Hamiltonian
formulation and its higher-codimensional aspect. The former is
called the Walker-Wang model (cf \cite{walker-wang-model} and
\cite{mathoverflow/194633}). However, while the relation between
CY and the Walker-Wang model are commonly believed, to the best
knowledge of the author, it has not been explicitly proved.

The later, on the other hand, is currently studied by A.
Kirillov's school starting around 2018. In particular, Tham and
Kirillov correctly defined CY model in (co)dimension $2$,
computed some examples, and proved the excision property, which
connects the CY model to factorization homology. Similar work can
be found in the context of skein modules (cf
\cite{cooke2019excision}).

In the rest of this section, we recall the results from Kirillov
and Tham. The main statement of this paper will follow in the
next section.

\begin{center}
  \begin{tabular}{ | c | c | c | c | c | c | }
    \hline
    $\Sigma$ & Disk & Cylinder & Sphere & 1-punctured torus & General \\ \hline
    $CY_{C}(\Sigma)$ & $C$ & $Z(C)$ & $Mu(C)$ & $Z^{el}(C)$ & $Z_{\sigma}(C)$ \\ \hline
        & & Drinfeld center & Muger center & Elliptic center &
                                                               \begin{tabular}{@{}c@{}}Categorical center \\ of higher genera \end{tabular} \\ \hline
  \end{tabular}
\end{center}

\begin{theorem}\cite[Section 5]{fac-homo--kirillov-tham}\label{thm/cy-of-a-disk}
  Let $C$ be a premodular category, and $\Sigma = D^{2}$ be the open
  disk. Then as abelian categories
  $$CY_{\Sigma}(C) \simeq C.$$
\end{theorem}

\begin{theorem}[Drinfeld center]\cite[Example 8.2]{fac-homo--kirillov-tham}\label{thm/drinfeld-center-and-cylinder}
  Let $C$ be a premodular category, $Z(C)$ its Drinfeld
  categorical center, and $\Sigma = S^{1} \times I = S^{1} \times (0,1)$ be
  the cylinder. Then as abelian categories
  $$CY_{\Sigma}(C) \simeq Z(C).$$
  Moreover, as multifusion categories, the topological nature of
  $\Sigma$ induces a so called reduced tensor product $\overline{\otimes}$
  (\cite{reduced--tham}, \cite{wasserman2019reduced}) for $Z(C)$.
  Indeed, stacking two cylinders together produces another
  cylinder $S^{1} \times (0,2) \simeq S^{1} \times I$.
\end{theorem}

\noindent Notice that the reduced tensor product is in general
different from the usual tensor product of the Drinfeld center.

\begin{theorem}[Excision principle]\cite[Theorem
  2.3]{fac-homo--kirillov-tham}\label{thm/excision-principle}

  Denote the Deligne tensor product by $\boxtimes$. Let $\Sigma$ be an open
  surface with $n$ punctures. Then the category $CY_{\Sigma}(C)$ has a
  structure of module category over $CY_{\Sigma}(S^{1} \times I)^{\boxtimes n}$,
  which is $(Z(C),\overline{\otimes})^{\boxtimes n}$ by
  \ref{thm/drinfeld-center-and-cylinder}.

  Let $\Sigma_{1}$ and $\Sigma_{2}$ be smooth oriented surfaces possibly
  with punctures. And let $CY$ be Crane-Yetter theory in
  dimension two. Then we have an equivalence of abelian
  categories.
  \[
    CY_{(\Sigma_{1} \cup \Sigma_{2})} \simeq CY_{\Sigma_{1}} \boxtimes_{CY_{({\Sigma_{1} \cap \Sigma_{2}})}} CY_{\Sigma_{2}},
  \]
  where $\boxtimes_{D}$ denotes the balanced (Deligne) tensor product
  over $D$ \cite{balanced-tensor-product}.
\end{theorem}

\begin{remark}
  For a related result on the excision principle, see
  \cite{cooke2019excision} and
  \cite{integrating-quantum-groups-over-surfaces}.
\end{remark}

\begin{theorem}[Muger centralizer]\cite[Corollary 8.5]{fac-homo--kirillov-tham}
  Let $C$ be a premodular category, $Z'(C)$ its Muger
  centralizer, and $\Sigma = S^{2}$ be the 2-sphere. Then as
  abelian categories
  $$CY_{\Sigma}(C) \simeq Z'(C).$$
  In particular, if $C$ is modular, then the result trivializes
  as in
  $$CY_{\Sigma}(C) \simeq (Vect).$$
\end{theorem}

Tham defined for a premodular category $C$ an associated category
$Z^{el}(C)$, coined the elliptic Drinfeld center. Its objects are
the triples $(X,\gamma_{1},\gamma_{2})$, where $X$ is an object of $C$ and
the $\gamma_{i}$'s are half-braidings of $X$ that satisfy certain
relations. See \cite{elliptic--tham} for a full definition. The
name is justified by the following theorem.

\begin{theorem}[elliptic Drinfeld center]\cite[Corollary 9.5+6]{fac-homo--kirillov-tham}\label{thm/elliptic-center-and-punctured-torus}
  Let $C$ be a premodular category, $Z^{el}(C)$ its elliptic
  Drinfeld center, and $\Sigma = \Sigma_{1,0}$ be a
  once-punctured torus. Then as abelian categories
  $$CY_{\Sigma}(C) \simeq Z^{el}(C).$$
  In particular, if $C$ is modular, then there is an equivalence
  of left $CY_{S^{1} \times I}(C)$-modules
  $$C \simeq CY_{D^{2}}(C) \simeq CY_{\Sigma}(C).$$
\end{theorem}

\begin{theorem}\cite[Corollary 4.5]{reduced--tham}
  Let $C$ be a premodular category, $Z$ the Drinfeld center
  construction, $\overline{\otimes}$ the stacking tensor product, and
  $\Sigma = S^{1} \times S^{1}$ be the standard torus. Then
  $$CY_{\Sigma}(C) \simeq Z((Z(C), \overline{\otimes})).$$
\end{theorem}

One of the main theorem of \cite{fac-homo--kirillov-tham} is that
Crane-Yetter theorem in dimension two also trivializes when the
input data is modular.

\begin{theorem}\cite[Remark 9.8]{fac-homo--kirillov-tham}\label{thm/cy-trivializes-over-mtc}
  Let $C$ be a modular category, and $\Sigma$ an $n$-punctured surface
  of genus $g$. Then up to equivalence $CY_{\Sigma}(C)$ is independent
  of $g$. In fact, we have an equivalence of module categories
  over $(Z(C),\overline{\otimes})^{\boxtimes n}$
  $$CY_{\Sigma}(C) \simeq C^{\boxtimes n}.$$
  Notice that when $n=0$ the power is the category of finite
  dimensional vector spaces.
\end{theorem}

An easy proof of this fact due to the author of this paper uses
the excision principle and a basic equivalence
$C \boxtimes C^{bop} \simeq Z(C)$.

\subsection{Main result}

\subsubsection{$CY \simeq Z$}

The main result of this paper is an explicit calculation of the
Crane-Yetter theory for all smooth oriented surfaces with at
least one puncture. In this section, we overview the statement
and its consequences, leaving a detailed proof to section
\ref{section/proof-for-the-main-statement}.

Crane-Yetter theory in dimension two is defined as a linear
category whose spaces of morphisms are vector spaces presented by
many complicated generators and relations. By a calculation we
mean to describe it as a category whose description is much
smaller. An analogy of this is that the first homology of the
circle ``calculated'' to be the group of integers
$$H_{1}(S^{1}) \simeq \mathbb{Z}.$$

Given any open surface $\Sigma$, we present it by an oriented
$2$-disk and some segments of its boundary glued. Which segments
are glued together are described by some combinatorial data,
called the admissible gluings. With a choice of an admissible
gluing $\sigma$, the categorical center of higher genera $Z = Z_{\sigma}$
is a specific category explicitly constructed in
\ref{def/categorical-center-of-higher-genera}, with some basic
properties given in \ref{subsection/properties-of-Z}. Roughly,
given a premodular category $C$ and an admissible gluing $\sigma$ of
rank $n$, the categorical center of higher genera $Z = Z_{\sigma}(C)$
is a category with objects of the form $(X,\gamma_{1},\ldots,\gamma_{n})$ where
$X$ is a $C$-object and the $\gamma_{i}$'s are half-braidings
satisfying specific relations. The main statement is that $Z$ is
equivalent to $CY(\Sigma)$ as a finite semisimple abelian category.

\begin{center}
  \includegraphics[height=2cm]{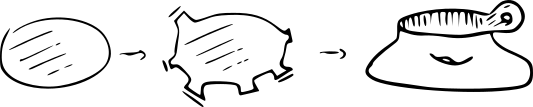}
\end{center}

\begin{theorem}[Main statement]\label{mirror/theorem/main-statement}
  Let $n$ be a nonnegative integer, $\sigma$ an admissible gluing of
  rank $n$, $\Sigma_{\sigma}$ the surface constructed from $\sigma$, $C$ a
  premodular category, $CY_{\Sigma_{\sigma}}(C)$ the Crane-Yetter theory of
  $\Sigma_{\sigma}$ depending on $C$, and $Z_{\sigma}(C)$ the categorical center
  of higher genera of $C$ with respect to $\sigma$. \\

  \noindent Then we have an equivalence of finite semisimple
  abelian categories
  $$
  CY_{\Sigma_{\sigma}}(C) \simeq Z_{\sigma}(C).
  $$
\end{theorem}

\noindent A detailed proof of the main statement is given in
section \ref{section/proof-for-the-main-statement}.

\begin{example}[$n=0$]
  For $n=0$, the surface $\Sigma$ is the open disk, the categorical
  center of higher genera reduces to the underlying premodular
  category $C$, so the theorem recovers that $CY_{\Sigma}(C) \simeq C$ as
  shown in \ref{thm/cy-of-a-disk}.
\end{example}
\begin{example}[$n=1$]
  For $n=1$, the only possible surface is the cylinder, the
  categorical center of higher genera reduces to the Drinfeld
  center $Z(C)$. Hence, the theorem recovers that
  $CY_{\Sigma}(C) \simeq Z(C)$ as shown in
  \ref{thm/drinfeld-center-and-cylinder}.
\end{example}
\begin{example}[$n=2$]
  For $n=2$, there are
  two possible surfaces: the $1$-punctured torus and the
  $3$-punctured disk. In the former case, the categorical center
  of higher genera reduces to the elliptic center $Z^{el}(C)$, so
  the theorem recovers that $CY_{\Sigma}(C) \simeq Z^{el}(C)$ as shown in
  \ref{thm/elliptic-center-and-punctured-torus}. In the later
  case, the theorem provides a new result.
\end{example}

\subsubsection{Remarks}

\begin{remark}\label{remark/shape-algebra-duality}
  In $H_{1}(S^{1}) \simeq \mathbb{Z}$, one sees the algebra of the
  shape $S^{1}$ and the shape of the algebra $\mathbb{Z}$. Our
  main result should be viewed as a higher analogue. That is, one
  sees the (higher) algebra of the shape $\Sigma_{\sigma}$ and the shape of
  the (higher) algebra $Z_{\sigma}$.
\end{remark}

\begin{remark}
  A full definition for premodular categories is given in
  \ref{def/premodular-category}.
\end{remark}

\begin{remark}
  These categories have their tensor structures and module
  categorical structures \cite{egno/tensor-cats} coming from
  their topological nature. This will be treated in the author's
  following work.
\end{remark}

\begin{remark}
  The smoothness condition is not necessary for our theory, but
  is included for the sake of simplicity. Indeed, later we will
  see that the Crane-Yetter theory in dimension $2$ can be
  defined based on stringnets. With the smooth structure, it is
  easier to regulate how they meet each other. On the other hand,
  Crane-Yetter theory works also in the PL-setting, parallel to
  its $3$-dimensional analogue, the Turaev-Viro theory. Curious
  readers are refer to a setup given in
  \cite{kirillov-balsam/turaev-viro-I}.
\end{remark}

\subsection{Summary of each section}

\begin{itemize}
  \item Section \ref{section/preliminaries}: Premodular
        categories and their graphical calculus are treated with
        examples and useful facts.
  \item Section \ref{section/topological-theory}: The relevant
        topological theory, namely the Crane-Yetter theory in
        dimension two, is treated formally in terms of
        string-nets (\cite{kirillov/string-net}).
  \item Section \ref{section/algebraic-theory}: The relevant
        algebraic theory, namely the categorical center of higher
        genera $Z_{\sigma}(C)$, is constructed. We prove some of its
        basic properties, such as its finite semisimple
        abelianess and its ambidextrous adjunction with the
        underlying $C$.
  \item Section \ref{section/proof-for-the-main-statement}: The
        proof for the main theorem is given, which bridges the
        topological theory and the algebraic theory.
  \item Section \ref{section/outlook-and-remarks}: Outlook and
        remarks.
  \item Appendix: For completeness, we include formal
        definitions, propositions, and detailed proofs.
\end{itemize}

\section{Preliminaries} \label{section/preliminaries}

\subsection{Premodular Category}

A full definition of a premodular category from scratch is
tedious but well known to experts. However, for completeness we
include it in the appendix
\ref{section/tensor-category}. In short, it is a
fusion category with a braiding and a spherical structure.

\begin{definition}[Premodular category]
  A premodular category is a braided fusion category equipped
  with a spherical structure.
\end{definition}

\begin{example}[Finite group]\label{example/premodular-category--rep-G} Let $G$
  be a finite group. Then the category $Rep(G)$ of
  finite-dimensional linear representations of $G$ over
  $\mathbb{k}$ has a natural structure of a premodular category.
\end{example}

\begin{example}[Drinfeld double]\label{example/premodular-category--rep-DG} Let
  $G$ be a finite group and $D(G)$ its Drinfeld double over
  $\mathbb{k}$. Then the category $Rep(D(G))$ of
  finite-dimensional linear representations of $D(G)$ over
  $\mathbb{k}$ has a natural structure of a premodular category.
\end{example}

\begin{example}[Crossed module] \label{example/premodular-category--rep-X}
  Let $\mathfrak{X}$ be a finite $2$-group (or called a finite
  crossed-module) \cite{bantay/chars-xmods}. Then the category
  $Rep(\mathfrak{X})$ of finite-dimensional linear
  representations of $\mathfrak{X}$ over $\mathbb{k}$ has a
  natural structure of a premodular category.
\end{example}

\begin{remark}
  Let $G$ be a finite group. Both $G$ and $D(G)$ can be viewed as
  special cases of finite crossed modules. Hence,
  \ref{example/premodular-category--rep-X} generalizes
  \ref{example/premodular-category--rep-G} and
  \ref{example/premodular-category--rep-DG}.
\end{remark}

\begin{example}[Quantum group] \label{example/premodular-category--quantum-group}
  In the case $\mathbb{k} = \mathbb{C}$, let $\mathfrak{g}$ be a
  semisimple Lie algebra and $q$ a root of unity. The
  semisimplified category $Rep(U_{q}(\mathfrak{g}))$ of the
  category of finite-dimensional representations of the quantum
  group $U_{q}(\mathfrak{g})$ has a natural structure of a
  premodular category.
\end{example}

\begin{example}[Even part of the quantum
  $\mathfrak{sl}_{2}$] \label{example/premodular-category--even-part-of-quantum-sl2}
  In the case $\mathbb{k} = \mathbb{C}$, let $q$ be a root of
  unity. The semisimplified category
  $C = Rep(U_{q}(\mathfrak{sl}_{2}))$ of the category of
  finite-dimensional representations of the quantum group
  $U_{q}(\mathfrak{sl}_{2})$ has a structure of a modular
  category. The even part $C_{0}$ of $C$ has a structure of a
  premodular category \cite{q-mackay}.
\end{example}

\subsection{Graphical Calculus}

We will use the technique of graphical calculus while dealing
with premodular categories. For a pedagogical exposition, see for
example (\cite{kirillov/mtc} and \cite{quantum-group--kassel}).
Notice however that we draw the morphisms as ribbon tangles in
the downward direction.

\begin{center}
  \includegraphics[height=5cm]{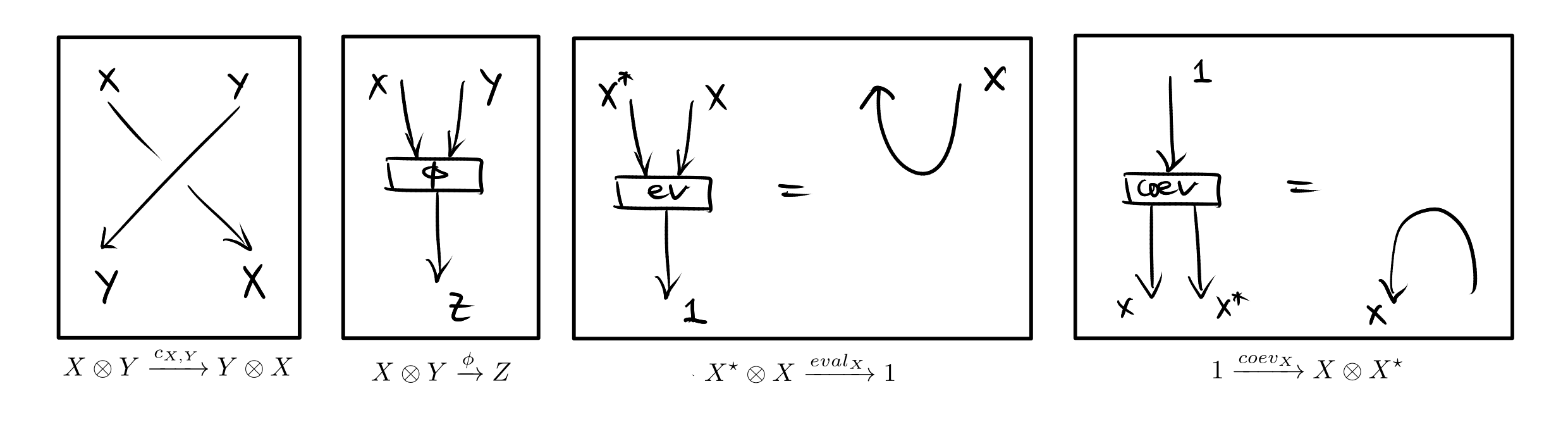}
\end{center}

\noindent An advantage of this is that many equalities among
morphisms can be proved graphically, thanks to \cite[Theorem 2.3.10]{kirillov/mtc}. For
example, to prove
$$eval_{Y} \circ c_{X,Y} \circ c_{X,Y^{\star}} \circ c_{X,Y} \circ coev_{Y} = c_{X,Y},$$
it suffices to establish an isotopy of ribbon tangles, which is
an obviously trivial task:

\begin{center}
  \includegraphics[height=3cm]{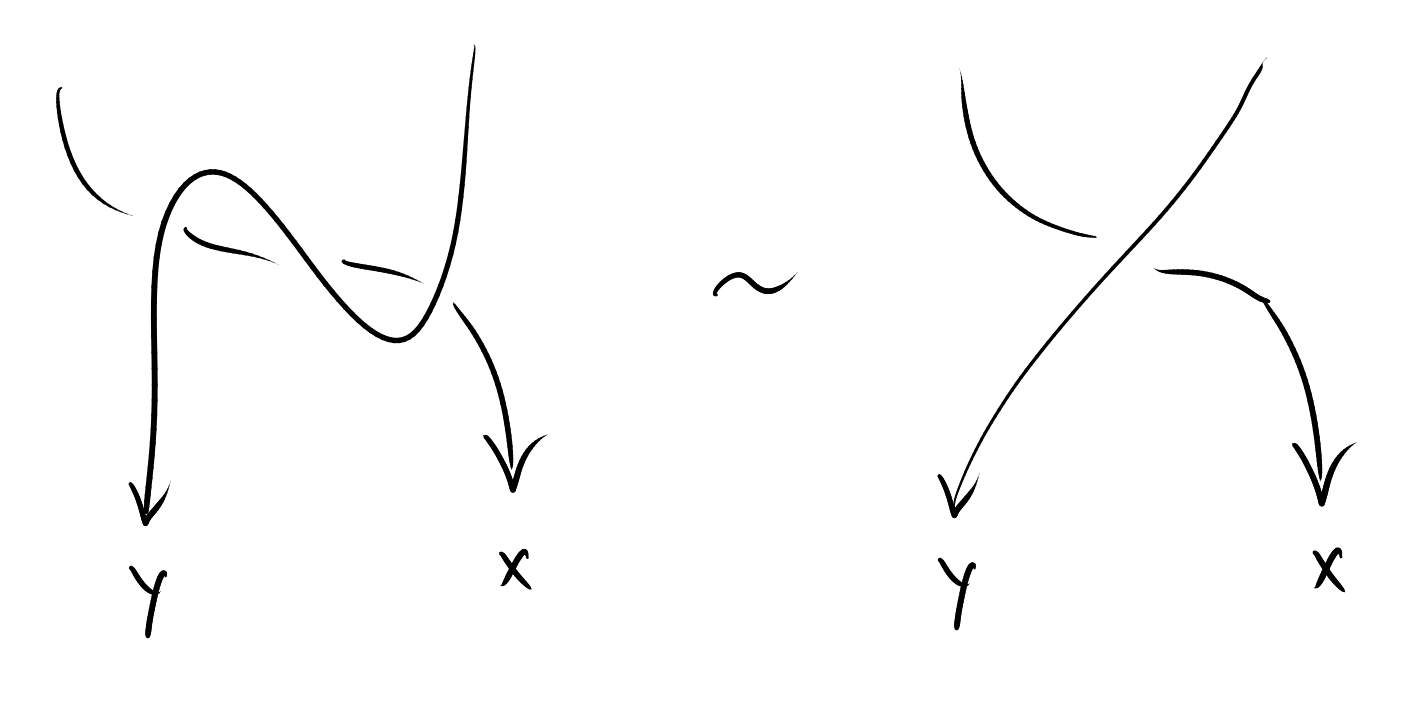}
\end{center}

\noindent In the rest of the section, we provide some useful
lemmas and notations.

\begin{lemma}\label{lemma/dual-of-hom-space}
  Let $C$ be a premodular category with spherical structure $a$.
  Let $X, Y$ be $C$-objects. Define a pairing of
  $\mathbb{k}$-linear spaces
  $$Hom_{C}(X,Y) \otimes Hom_{C}(Y,X) \xrightarrow{(,)} \mathbb{k}$$
  that sends $\phi \otimes \psi$
  to $$eval_{B^{\star}} \circ (a_{B} \otimes 1_{X}) \circ (\phi \otimes \psi) \circ coev_{A} \in End_{C}(\mathbb{1}) \simeq \mathbb{k}.$$
  Then the pairing is nondegenerate by the semisimplicity of $C$.
  Moreover, $Hom_{C}(X^{\star},Y^{\star}) \simeq Hom_{C}(Y,X)$ by the rigidity
  of $C$, so $Hom_{C}(X^{\star},Y^{\star})$ can be naturally realized as
  the dual vector space of $Hom_{C}(X,Y)$.
\end{lemma}

Define
$$\omega_{X,Y} := \Sigma_{i} \phi_{i} \otimes \phi^{i} \in Hom_{C}(X,Y) \otimes Hom_{C}(Y,X)$$
where the $\phi_{i}$'s and the $\phi^{i}$'s form a pair of an
orthonormal basis and a dual basis respectively for
$Hom_{C}(X,Y)$ and $Hom_{C}(Y,X)$ under the identification given
in \ref{lemma/dual-of-hom-space}. Graphically, we use a dummy
variable as a short-hand notation:

\begin{center}
  \includegraphics[height=5cm]{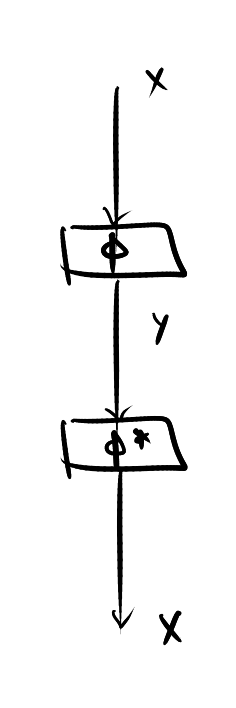}
\end{center}

\begin{lemma} \label{lemma/fundamental-lemma-of-Omega}
  Let $C$ be a premodular category, $a$ its spherical structure,
  and $W$ a $C$-object. Then
  $$1_{W} = \Sigma_{i \in \mathcal{O}(C)} \Sigma_{l} dim_{a}(i) \phi^{l} \circ \phi_{l},$$
  where the $\phi_{l}$'s and the $\phi^{l}$'s form an orthonormal basis
  and a dual basis respectively for $Hom_{C}(W,i)$ and
  $Hom_{C}(i,W)$.
\end{lemma}

\begin{notation}[$\Omega$]\label{def/Omega}
  Let $C$ be a premodular category, $a$ be its spherical
  structure, and $O(C)$ be the set of isomorphism classes of
  simple objects of $C$. We use $\Omega$ in graphics to represent
  $\oplus_{i \in \mathcal{O}(C)} dim_{a}(i) i$. We also denote
  $dim(\Omega)$ by
  $\Sigma_{i \in \mathcal{O}(C)} dim_{a}(i)^{2}$.
\end{notation}

\noindent With this shorthand notation $\Omega$, we present the
lemma graphically by

\begin{center}
  \includegraphics[height=4cm]{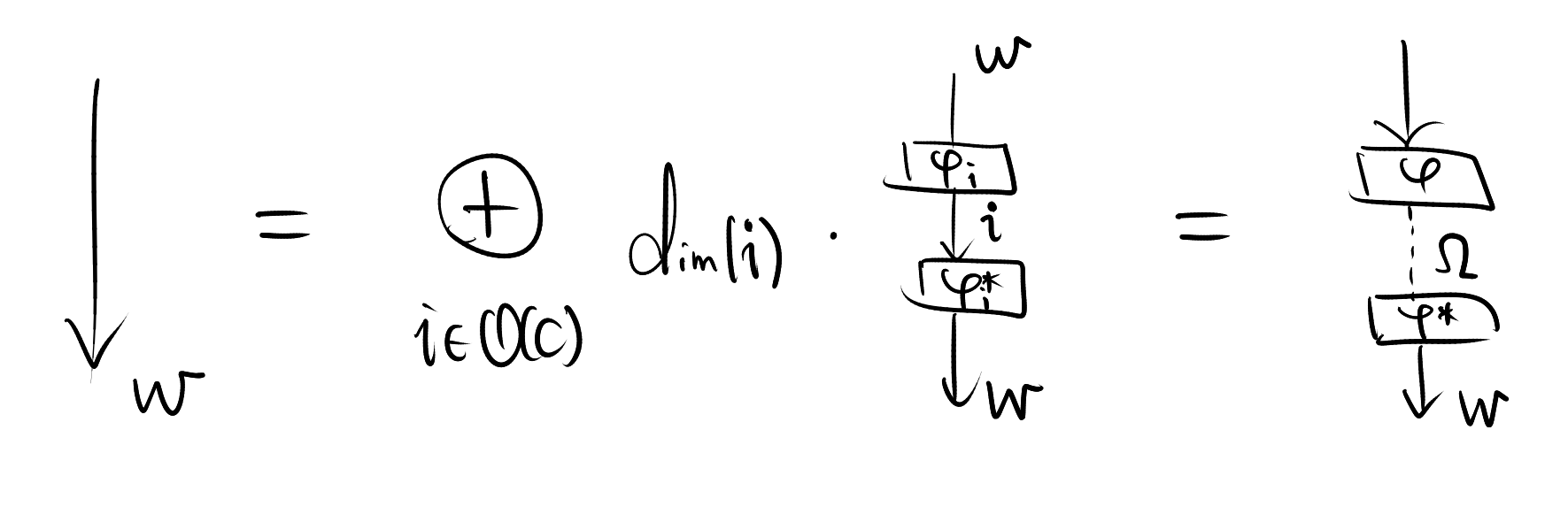}
\end{center}

\section{Topological theory} \label{section/topological-theory}

In this section, we describe the topological side of our main
statement (cf \ref{mirror/theorem/main-statement} and
\ref{remark/shape-algebra-duality}), namely the Crane-Yetter
theory in dimension two, is treated formally in terms of string
nets. This includes a definition of Crane-Yetter in dimension
two, and a combinatorial description of oriented smooth surfaces.
The former requires the notion of string nets (also called tensor
nets or tensor networks), which will be treated in
\ref{section/string-nets}. A definition of Crane-Yetter theory in
dimension two follows in
\ref{section/crane-yetter-theory-in-dimension-two}. Finally, the
combinatorial description of smooth surfaces
($\sigma$-construction) is given in
\ref{section/sigma-construction}.

\subsection{String nets} \label{section/string-nets}

Originated from Penrose combinatorial description of space-time
\cite{penrose/angularmomentum}, string nets are the building
stone of Crane-Yetter theory. They are also called (quantum)
tensor nets or tensor networks in other contexts. In dimension
two, they were first explicitly written by the physicists Levin
and Wen in \cite{levin-wen-model}. For Crane-Yetter theory,
however, we need string nets in dimension three. Following
\cite{fac-homo--kirillov-tham}, we provide a formal definition
\ref{def/string-nets-in-3d} of them in this section.

Before the formal definition, keep in mind that it aims to
formalizes the pictures of the following sort.

\begin{center}
  \includegraphics[height=9cm]{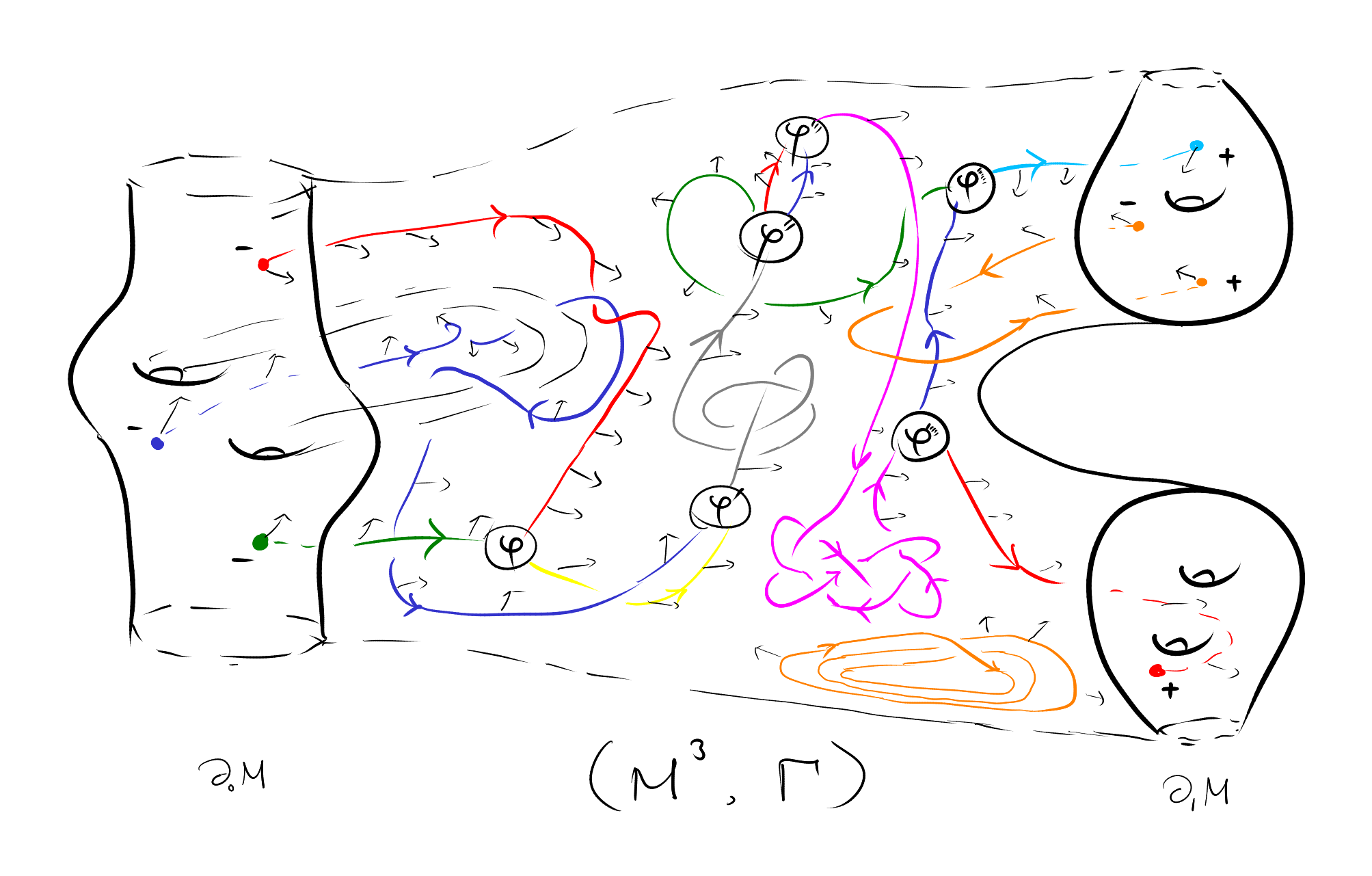}
\end{center}

\begin{definition}[$2$-folds]\label{def/2-fold-convention}
  A $2$-fold is either a compact oriented smooth manifold without
  boundary of real dimension $2$, or such a manifold with
  finitely many points removed (punctures). A $2$-fold is also
  called a surface.
\end{definition}

\begin{definition}[Extended $2$-folds]\label{def/extended-2-fold}
  An extended $2$-fold is a $2$-fold $M$ with the extra data
  $$\{(p_{1},v_{1},or_{1}) \ldots (p_{n},v_{n},or_{n})\},$$ where
  $n < \infty$, the points $p_{i} \in M$ are disjoint to each other, the
  tangent vectors $v_{i} \in T_{p_{i}}M$ are nonzero, and the
  orientations $or_{i}$ are in the set $\{+,-\}$.
\end{definition}

\begin{definition}[$3$-folds]\label{def/3-fold-convention}
  A $3$-fold is an oriented smooth manifold with boundary of real
  dimension $3$.
\end{definition}

\begin{definition}[Framed arcs in a
  $3$-fold]\label{def/framed-arcs-in-a-3-fold}
  Let $M$ be a $3$-fold (\ref{def/3-fold-convention}). An arc $\alpha$
  in $M$ is a smooth embedding of the standard interval
  $I = [0,1]$ (with orientation from $0$ to $1$) into $M$. We
  require that if an end-point is sent by $\alpha$ to the boundary
  $\partial M$, then $\alpha$ has to intersect the boundary transversally.

  A framing of an arc $\alpha$ in $M$ is a non-vanishing smooth
  section $s$ of the normal bundle of $\alpha(I) \subseteq M$. A framed arc is
  an arc equipped with a framing.
\end{definition}

\begin{remark}
  For our theory, the smoothness condition is not necessary but
  included for the sake of simplicity. Crane-Yetter theory works
  also in the PL-setting, parallel to its $3$-dimensional
  analogue, the Turaev-Viro theory. Curious readers are referred to
  the setup given in \cite{kirillov-balsam/turaev-viro-I}.

  The framing, on the other hand, is necessary. Such structure is
  expressed in slightly different way in related works. For
  example, in the context of skein modules, people use the notion
  of ribbons instead of that of arcs. The ``width'' of a ribbon
  corresponds to the normal vector from the section.
\end{remark}

\begin{definition}[Framed graphs in a
  $3$-fold]\label{def/framed-graphs-in-a-3-fold}
  Let $M$ be a $3$-fold (\ref{def/3-fold-convention}). A framed graph
  $\Gamma$ in $M$ is a finite collection of framed arcs $\alpha_{i}$
  \ref{def/framed-arcs-in-a-3-fold} satisfying the following
  conditions.

  \begin{itemize}
    \item Denote the set of arcs of $\Gamma$ by $E(\Gamma)$.
    \item The images of the embeddings $\alpha_{i}$'s do not meet each
          other, with an exception at their endpoints which must
          be in the interior of $M$.
    \item Let $v$ be a point in the interior of $M$. Denote
          $Out(v)$ ($In(v)$, resp.) be the set of the $\alpha_{i}$'s
          with $v = \alpha_{i}(0)$ ($\alpha_{i}(1)$, resp.). Denote the set
          of all arcs (edges) $E(v) := In(v) \cup Out(v)$.
          The directions of the tangent vectors of the $\alpha$'s that
          end at $v$ must be different, i.e. for each $i \neq j$,
          there is no positive real number $r$ such that
          $v_{i} = rv_{j}$. In the case where $E(v)$ is nonempty,
          we call $v$ a vertex of $\Gamma$. Denote the set of vertices
          by $V(\Gamma)$.
    \item Let $\alpha \in E(\Gamma)$. If an end of $\alpha$ ends on the boundary
          of $M$, we call the end-point a boundary point of $\Gamma$.
          Denote the set of boundary points by $B(\Gamma)$.
  \end{itemize}
\end{definition}

\begin{definition}[Extended $3$-folds]\label{def/extended-3-fold}
  An extended $3$-fold is a pair of a $3$-fold $M$ and a framed graph
  $\Gamma$ in $M$.
\end{definition}

Notice that an extended $3$-fold $(M,\Gamma)$ naturally induces an
extended $2$-fold $\partial(M,\Gamma) := (\partial M, \partial \Gamma)$, where $\partial \Gamma$ denotes the
set
$$\{(p_{i},v_{i},or_{i}) \,|\, p_{i} \in V(\Gamma) \cap \partial M\},$$
where $v_{i}$ is naturally identified via
$$N_{p_{i}} = T_{p_{i}}M/T_{p_{i}}\alpha \simeq T_{p_{i}}(\partial M)$$
with the framing vector of $\Gamma$ at the point $p_{i}$, and
$or_{i} \in \{+,-\}$ is $+$ ($-$, resp.) if the framed arc $\alpha$ that
passes through $p_{i}$ is oriented such that $\alpha(1) = p_{i}$
($\alpha(0) = p_{i}$, resp.).

\begin{center}
  \includegraphics[height=5cm]{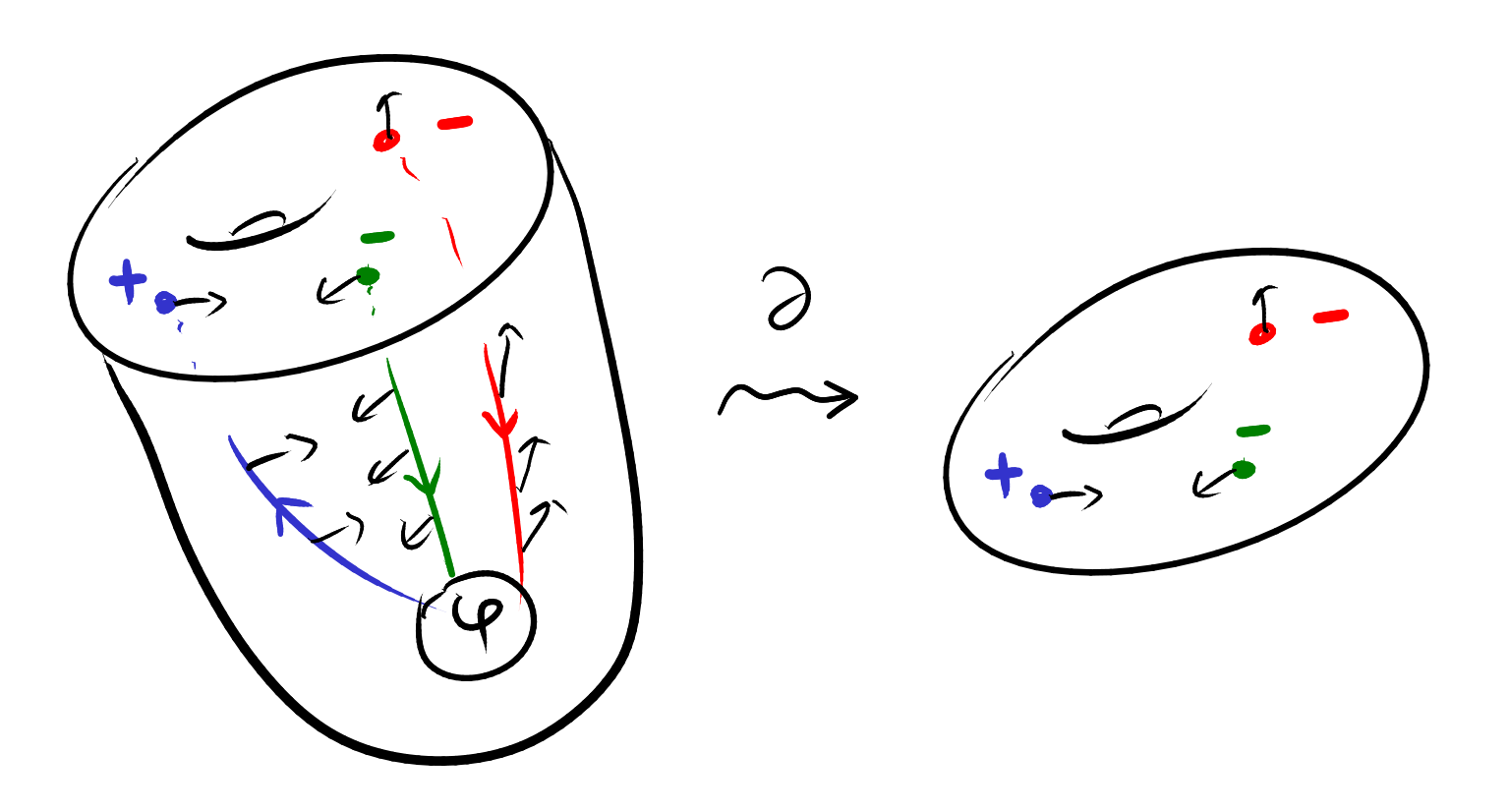}
\end{center}

\begin{definition}[$C$-extended $2$-folds]
  Given a premodular category $C$, a $C$-extended $2$-fold, or
  a $C$-colored extended $2$-fold, is an extended $2$-fold
  with an extra data: a $C$-object $X_{i}$ is assigned to each
  oriented framed point $(p_{i},v_{i},or_{i})$.
  We call $X_{i}$ the ``color'' assigned to the point $p_{i}$.
\end{definition}

To define a $C$-extended $3$-fold, we need the Reshetikhin-Turaev
theory (\cite{turaev-qiok-3-manifolds}, \cite{kirillov/mtc},
\cite{quantum-group--kassel}) for genus-$0$ surfaces which we
recall here. Let $C$ be a premodular category. To each $C$-extended
$2$-fold
$$M = (M, \{(X_{1}, p_{1}, v_{1}, or_{1}) \ldots (X_{n}, p_{n}, v_{n}, or_{n})\})$$
diffeomorphic to a sphere, the first part of Reshetikhin-Turaev
theory functorially assigns a vector space $RT(M)$ that is
(non-canonically) isomorphic to
$$\langle X_{1}^{\epsilon}, \ldots, X_{n}^{\epsilon} \rangle := Hom_{C}(\mathbb{1}, X_{1}^{\epsilon} \otimes \ldots \otimes X_{n}^{\epsilon}),$$
where $X_{i}^{\epsilon}$ denotes $X_{i}$ if $(or_{i} = +)$ or
$X_{i}^{\star}$ if $(or_{i} = -)$.

Let $C$ be a premodular category, and $V$ be a
real vector space of dimension $3$.
Let $S$ be a finite collection of distinct framed oriented rays from the
origin of $V$, with an assignment $S \xrightarrow{\phi} Obj(C)$. In
this case, we say $V$ has a finite collection of distinct $C$-colored
rays. Then the Reshetikhin-Turaev theory for genus-$0$
surfaces naturally assigns a vector space $RT_{C}(V,S,\phi)$ to
the sphere $(V-0)/\mathbb{R_{+}}$.

\begin{definition}[$C$-extended
  $3$-folds]
  Let $C$ be a premodular category, and $M$ be an extended
  $3$-fold $(M, \Gamma)$. A $C$-coloring of $(M, \Gamma)$ is an assignment
  as follows:
  \begin{itemize}
    \item To each arc $\alpha$ of $\Gamma$, assign a $C$-object $X(\alpha)$.
    \item After such assignment, to each vertex $p$ of $\Gamma$, the
          tangent space at $p$ naturally has a finite collection
          of $C$-colored rays $(S,\phi)$.
    \item The Reshetikhin-Turaev theory for genus-$0$ surfaces
          assigns a vector space $RT(p):=RT_{C}(T_{p}(M),S,\phi)$ as
          above. Note that $RT(p)$ is (non-canonically)
          isomorphic to $$Hom_{C}(\otimes X_{i}, \otimes X_{o})$$ where the
          $X_{i}$ runs through the objects assigned to all
          incoming arcs, and the $X_{o}$ runs through the objects
          assigned to all outgoing arcs.
    \item After such assignment, to each vertex $p$ of $\gamma$ assign
          a vector $v \in RT(p)$.
  \end{itemize}
  A $C$-extended $3$-fold $M$ is a $3$-fold with a $C$-colored
  graph inside. This gives the boundary $\partial M$ a $C$-extended
  surface structure. Conversely, we call such a $C$-colored graph
  a framed graph that satisfies the boundary condition posed by the
  $C$-extended surface $\partial M$.
\end{definition}

Let $C$ be a premodular category. While the first part of
Reshetikhin-Turaev theory for genus-$0$ surfaces assigns a vector
space to a $C$-extended genus $0$ surface, the second part of it
assigns a $C$-extended $3$-fold $(M,\Gamma)$ diffeomorphic to a ball
to a vector $RT(M,\Gamma) \in RT(\partial M)$.

\begin{definition}[String nets in $3$D]\label{def/string-nets-in-3d}
  Let $C$ be a premodular category and $M$ a $3$-fold whose
  boundary $\partial M$ is a $C$-extended surface. Let $F$ be the free
  vector space over $\mathbb{k}$ generated by all $C$-colored
  graphs that satisfy the boundary condition posed by $\partial M$. Let
  $N$ be the subspace generated by either of the following
  \begin{itemize}
    \item The difference $\Gamma - \Gamma'$ of two $C$-colored graphs that
          are smoothly isotopic to each other.
    \item A linear combination $v = \Sigma c_{i} \Gamma_{i}$ such that
          there exists a closed region $B \subseteq M$ diffeomorphic to a
          ball such that the vector assigned by the
          Reshetikhin-Turaev theory for genus-$0$ surfaces of
          $v|_{B}$ is the zero vector.
  \end{itemize}
  We call $S(M):=F/N$ the space of string nets of $M$ with the
  given boundary condition, and we call an element of $S(M)$ a
  string net.
\end{definition}

\noindent We conclude this section with a useful lemma.

\begin{lemma}[Sliding lemma]\label{lemma/sliding-lemma}
  Let $C$ be a premodular category. Then the following string
  nets are equal, where $\Omega$ is the shorthand notation given
  in \ref{def/Omega}. Heuristically, the moral is that $\Omega$
  protects anything ``inside'' it by making it transparent.
  \begin{center}
    \includegraphics[height=5cm]{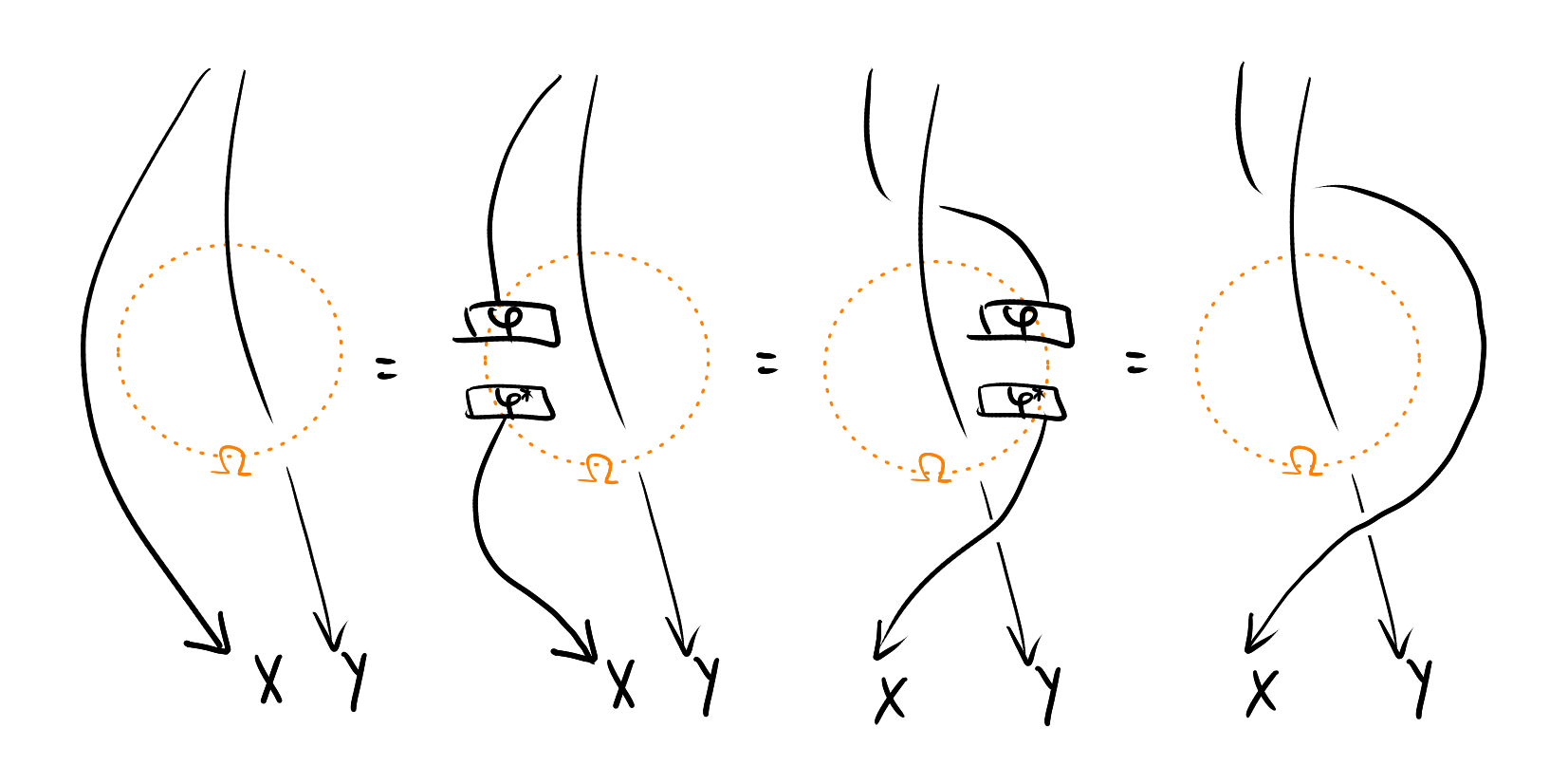}
  \end{center}
\end{lemma}

\begin{proof}
  Apply \ref{lemma/fundamental-lemma-of-Omega} locally with $W = X \otimes \Omega$.
  Use isotopy and the naturality of the braidings.
  And then apply \ref{lemma/fundamental-lemma-of-Omega} locally again with $W = \Omega \otimes X$.
\end{proof}

Notice that in fact $Y$ can be more general than an object - the
lemma works even when $Y$ is a puncture.

\subsection{Crane-Yetter theory in dimension two
  ($CY$)} \label{section/crane-yetter-theory-in-dimension-two}

We define the Crane-Yetter theory in dimension two in this
subsection, following \cite[section 5]{fac-homo--kirillov-tham}.
Let $\Sigma$ be a smooth oriented 2-manifold and $C$ a premodular
category. To define $CY_{\Sigma}(C)$, we first define an
auxiliary category $cy_{\Sigma}(C)$.

\begin{definition}[$cy_{\Sigma}(C)$, an auxiliary category]\label{def/cy-in-2d}
  Given a premodular category $C$ and a $2$-fold $\Sigma$, we define
  the $\mathbb{k}$-linear category $cy_{\Sigma}(C)$ as follows. An
  object is a collection $c$ of $C$-colored points and tangent
  vectors, such that $(\Sigma, c)$ is a $C$-extended $2$-fold. Given
  two objects $c$ and $c'$, the morphism space
  $Hom_{cy_{\Sigma}(C)}(c,c')$ between $c$ and $c'$ is defined to be
  the space of string nets for the $3$-fold $\Sigma \times [0,1]$
  satisfying the boundary condition $(\overline{c} \times \{0\}) \cup (c' \times \{1\})$,
  where $\overline{c}$ denotes the same collection of $C$-colored
  points as $c$ does but with all orientations flipped.
\end{definition}

Two examples of morphisms in $cy_{\Sigma_{1,0}}(C)$ is depicted as
follows.

\begin{center}
\includegraphics[height=3cm]{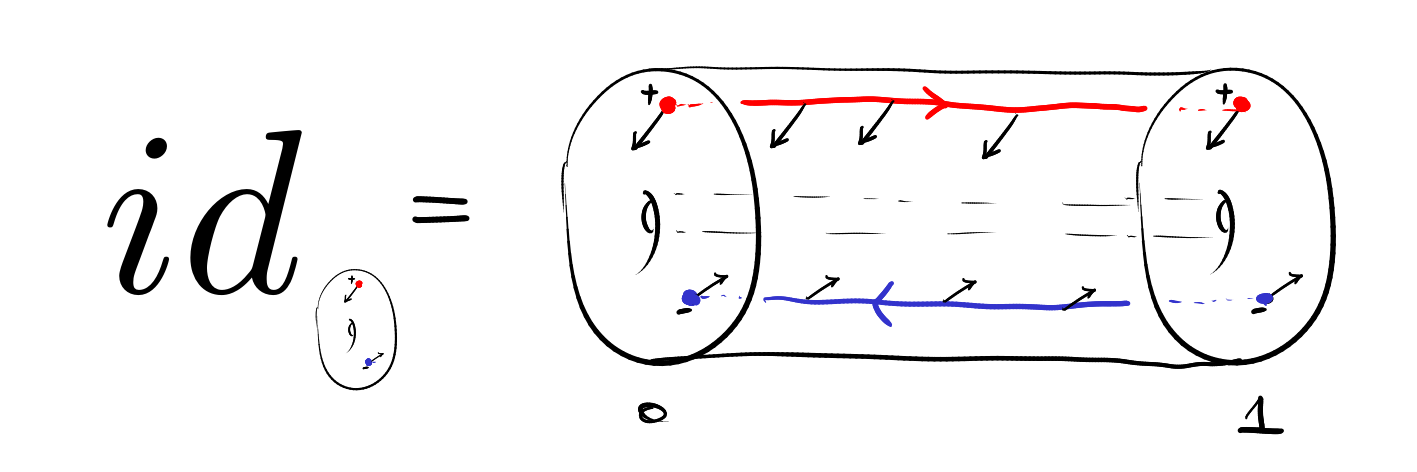}
\end{center}
\begin{center}
\includegraphics[height=3cm]{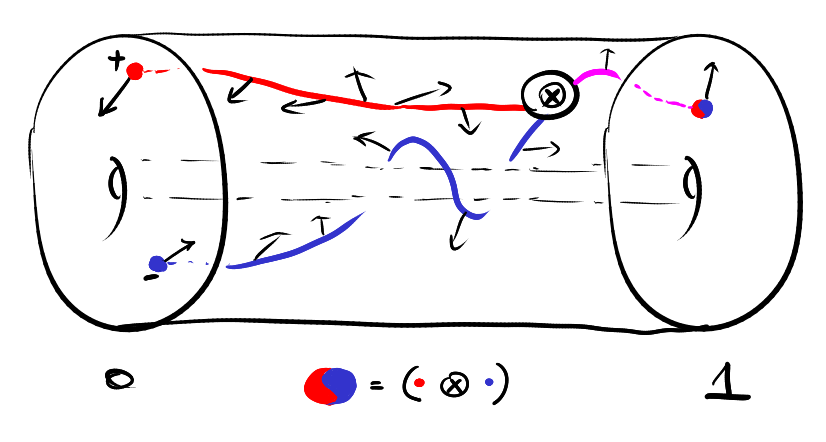}
\end{center}

\begin{definition}[Karoubi
  envelope]\label{def/karoubi-completion}
  Given an additive category $C$, its Karoubi envelope
  (Karoubi completion) $\Kar(C)$ is defined to be the category as
  follows. The objects are pairs $(X,p)$, where $X \in \Obj(C)$ and
  $p \in Hom_{C}(X,X)$, such that $p^{2}=p$. Given objects
  $\bar{X}=(X,p)$ and $\bar{Y}=(Y,q)$, the space of morphisms
  $Hom_{\Kar(C)}(\bar{X},\bar{Y})$ is defined to be the subspace
  of $Hom_{C}(X,Y)$ consisting of those $f$ such that $qfp = f$.
\end{definition}

\noindent The Karoubi envelope is the pre-abelian completion in
our context.

\begin{definition}[$CY_{\Sigma}(C)$, Crane-Yetter theory in dimension
  2]\label{def/crane-yetter-in-dimension-two}
  With the notations above, we define
  \begin{equation}
    CY_{\Sigma}(C) := \Kar(cy_{\Sigma}(C)).
  \end{equation}
\end{definition}

\begin{remark}
  The definition given in \ref{def/crane-yetter-in-dimension-two}
  was first given in \cite[section 5]{fac-homo--kirillov-tham}.
  That it extends the original Crane-Yetter theory is proved
  in \cite{tham/thesis}.
\end{remark}

It is immediate from the definition that $cy_{\Sigma}(C)$ is additive.
On the other hand, $CY_{\Sigma}(C)$ is in fact finite semisimple
abelian for all surfaces with at least one puncture (cf
\ref{thm/sigma-construction-subsumes-all-open-surfaces},
\ref{thm/cat-prop-of-Z}, \ref{main-statement}). It's conjectured
that it holds in fact for all surfaces.

\subsection{A presentation of surfaces
  ($\sigma$-construction)} \label{section/sigma-construction}

In this paper, we construct a surface $\Sigma$ from the standard disk
and an additional data $\sigma \in \Adm_{2n}$., give some examples, and
prove that such construction produces all oriented surfaces
with at least one puncture.

\begin{definition}[$Adm_{2n}$, admissible
  gluings]\label{def/adm-gluing}
  Let $n$ be a nonnegative integer. An element $\sigma$ in the
  permutation group $S_{2n}$ on $2n$ elements is called an
  admissible gluing (of rank $n$), if $\sigma$ satisfies the
  following conditions.
  \begin{itemize}
    \item $\sigma$ has no fixed points.
    \item $\sigma$ is an involution; i.e. $\sigma^{2} = 1$.
  \end{itemize}
  We denote the subset of admissible gluings by
  $\Adm_{2n} \subseteq S_{2n}$.
\end{definition}

\begin{definition}[$\Sigma_{\sigma}$, $\sigma$-construction]\label{def/sigma-construction}
  For each admissible gluing $\sigma \in \Adm_{2n}$, we construct
  a smooth surface $\Sigma_{\sigma}$. Start from the standard
  oriented disk. We choose $2n$ closed segments with the same
  length from the boundary. To make the presentation easier, we
  emphasize them by drawing them like legs (without changing the
  diffeomorphism type), and we call them legs from now on.

  \begin{center}
    \includegraphics[height=3cm]{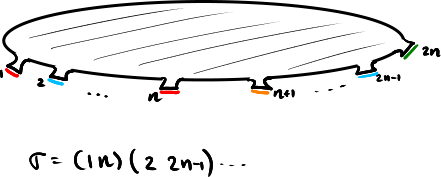}
  \end{center}

  \noindent Glue the end of the legs in pairs according to
  $\sigma$ with the orientation preserved. Finally, removed the
  boundary the the result to be an open surface. The result is
  denoted by $\Sigma_{\sigma}$.
\end{definition}

\begin{example}
  The only element $(1 2)$ in $\Adm_{2}$ constructs the cylinder
  $\Sigma_{(1 2)} \simeq \mbox{Cylinder}.$

  \noindent The elements $\sigma_{0,3} = (1 2)(3 4)$ and
  $\sigma_{1,1} = (1 3)(2 4)$ in $\Adm_{4}$ construct a 3-punctured
  sphere $\Sigma$ and a 1-punctured torus respectively.

  \begin{center}
    \includegraphics[height=4cm]{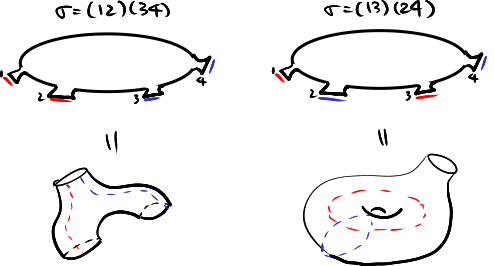}
  \end{center}
\end{example}

So constructed surfaces must have at least one puncture, thus the
procedure does not give all surfaces. However, the following
theorem shows that this is the only case it misses.

\begin{theorem}\label{thm/sigma-construction-subsumes-all-open-surfaces}
  The $\sigma$-constructions produce all oriented surfaces with
  at least one puncture (i.e. all open surfaces).
\end{theorem}
\begin{Proof}
  Indeed, the admissible gluing $\sigma_{1,1} = (13)(24) \in \Adm_{4}$
  gives an once-punctured torus $\Sigma_{\sigma_{1,1}}$. Similarly, the
  admissible gluing
  \[
    \sigma_{2,1} = \sigma_{1,1} \circ (57)(68) = (13)(24)(57)(68) \in \Adm_{8}
  \]
  gives an once-punctured surface of genus two. Following this
  fashion, for any $g \in \mathbb{N}$ one can construct an
  once-punctured surface of genus $g$ by using the admissible
  gluing
  \[
    \sigma_{g,1} = (13)(24)(57)(68) \ldots ((4g-3)(4g-1))((4g-2)(4g)) \in \Adm_{4g}.
  \]

  \begin{center}
    \includegraphics[height=10cm]{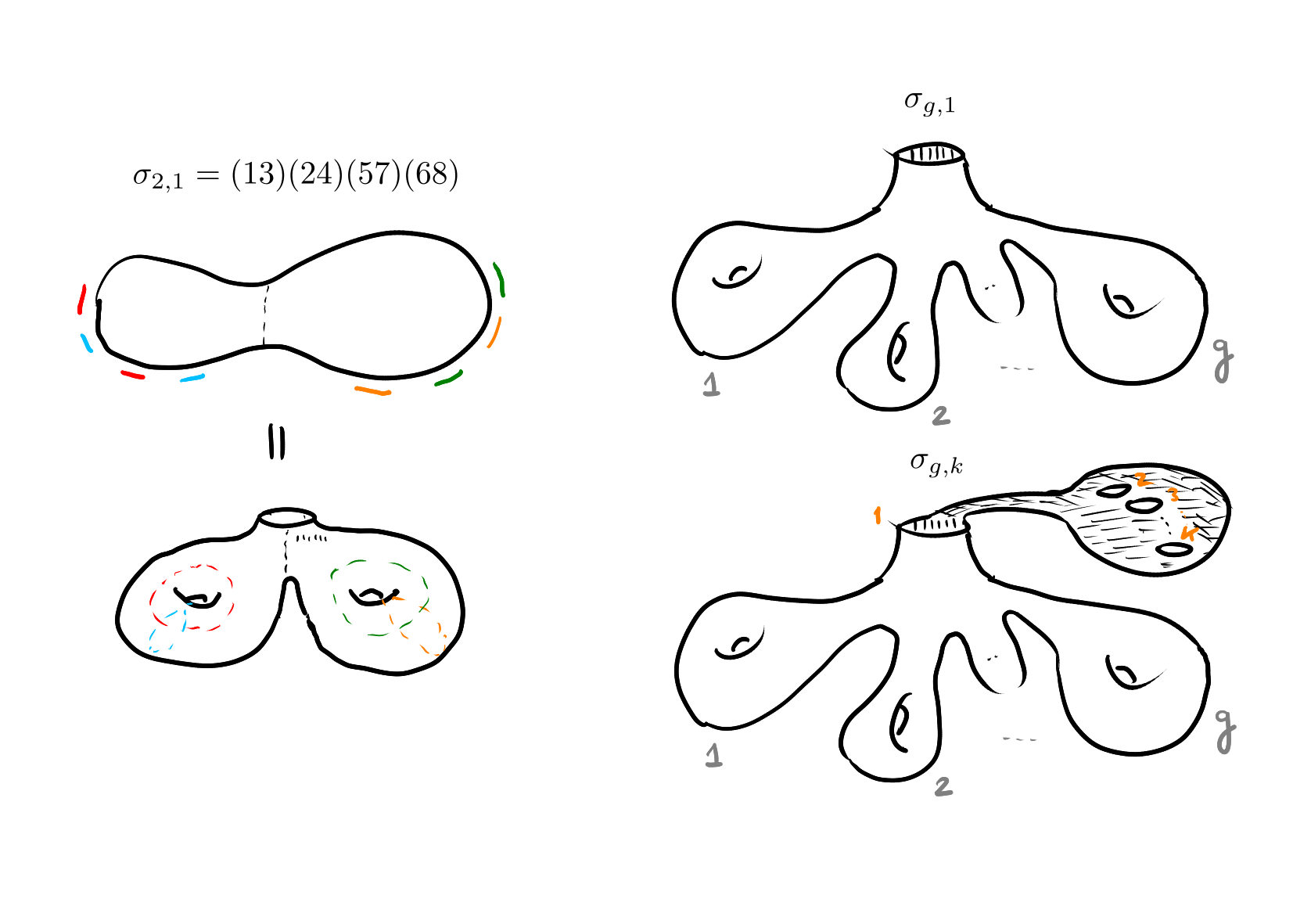}
  \end{center}

  \noindent To add $k$ punctures to the surface, use the
  admissible gluing
  \[
    \sigma_{g,(k+1)} = \sigma_{g,1} \circ ((4g+1)(4g+2))((4g+3)(4g+4)) \ldots ((4g+2k-1)(4g+2k)) \in \Adm_{4g+2k}.
  \]
  Then the statement follows from the well-known classification
  of oriented smooth surfaces.
\end{Proof}

\begin{notation}[$\sigma$-orbit]\label{notation/sigma-orbit}
  We take this opportunity to introduce a later useful notation.
  Let $\sigma \in \Adm_{2n}$. Denote $[i]$ to be the orbit
  of $$i \in \{ 1, 2 \ldots, 2n \}$$ under the action of $\sigma$, $[i]'$ the
  smaller number in the set $[i]$, and $[i]''$ the larger number
  in the set $[i]$. Note that the set $\{[1], [2], \ldots, [2n] \}$
  has exactly $n$ elements. \\

  \noindent As an example, for $\sigma = (13)(24)$, we have
  \begin{align*}
    [1] &= \{1, 3\}, \quad [1]' = 1, \quad [1]'' = 3; \\
    [2] &= \{2, 4\}, \quad [2]' = 2, \quad [2]'' = 4; \\
    [3] &= \{1, 3\}, \quad [3]' = 1, \quad [3]'' = 3; \\
    [4] &= \{2, 4\}, \quad [4]' = 2, \quad [4]'' = 4.
  \end{align*}
\end{notation}

\section{Algebraic theory} \label{section/algebraic-theory}

In this section, we describe the algebraic side of our main
statement (cf \ref{mirror/theorem/main-statement} and
\ref{remark/shape-algebra-duality}), namely the categorical
center of higher genera $Z$.

Its definition is quite algebraic and abstract, so some
motivation is supplemented in
\ref{section/motivation-for-center}. The formal definition is
given in \ref{section/categorical-center-of-higher-genera}.
Finally, some basic properties of $Z$ are proved in
\ref{subsection/properties-of-Z}. In particular, we show that $Z$
is finite abelian semisimple, and that there is a strictly
ambidextrous adjunction between $Z = Z(C)$ and the underlying
premodular category $C$.

\subsection{Motivation: Drinfeld categorical center} \label{section/motivation-for-center}

Abstract algebraic theories (groups, rings, modules.. etc) are
ubiquitous in modern mathematics. Among the algebraic objects,
the abelian ones are simpler, and are often first treated. One
then builds the theory toward the generic cases. In group theory,
for example, one can study a group $G$ by starting with its
center $Z(G) \subseteq G$ and then apply induction.

Drinfeld's categorical center is an analogue in the categorical
setting. There, algebras are replaced by categorical algebras
(more precisely, by monoidal categories \cite{egno/tensor-cats}),
and centers are replaced by categorical centers. As in the
classical theory, the theory of the one side helps that of the
other.

In contrast to the classical case, categorical centers need not
be smaller nor easier. This is due to the fact that equalities
are replaced by equivalences in the categorical settings.
Therefore, the condition $ab=ba$ is replaced by $ab \simeq ba$. That
is to say, a categorical commutativity not only remembers both
sides being identified, but also \textbf{how} they are
identified. Therefore, a typical object in the Drinfeld center
$Z(C)$ is a pair $(X \in \Obj(C), \gamma)$, where $\gamma$ is a half-braiding
that encodes how $X$ commutes with all the others. To be more
precise, a \textbf{half-braiding} $\gamma$ of $X$ is a natural
equivalence
$$(-) \otimes X \xrightarrow{\gamma} X \otimes (-)$$
satisfying some compatibility conditions \ref{def/half-braiding}.
It is worthwhile to mention that such construction has been
successful in many contexts, e.g. representation theory,
statistical physics, knot theory, .. etc.

Categorical center of higher genera
$Z = Z_{\sigma}(C) = Z_{\Sigma_{\sigma}}(C)$, on the other hand, generalizes the
Drinfeld center. Instead of remembering how $X$ commutes with
others, an object $(X,\gamma)$ remembers how $X$ commutes in
\textbf{multiple} different ways. The amount of ways depends on
the underlying surface $\Sigma = \Sigma_{\sigma}$. Therefore, an object of
$Z_{\sigma}(C)$ is a pair $(X,\gamma)$, where $\gamma$ is a collection of
half-braidings
$$
\gamma = \{ \gamma_{1}, \gamma_{2}, \ldots, \gamma_{n} \}.
$$
However, extra conditions must be carefully imposed in order to
keep track of the underlying topological data. In contrast to the
case of Drinfeld center, multiple half-braidings give essentially
infinite ways to fuse via tensors, e.g.
$\gamma_{1}\gamma_{2}\gamma_{2}\gamma_{1}\gamma_{1}\gamma_{1}\gamma_{2} \ldots$. Therefore, suitable
commutative relations among the half-braidings are needed. This
is given in the formal definition of $Z_{\sigma}(C)$ as \texttt{(:comm
  1)}, \texttt{(:comm 2)}, and \texttt{(:comm 3)} (cf
\ref{def/comm-relation}).

Before moving on to the formal definition of $Z_{\sigma}(C)$, let us
remark on the premodular condition on $C$. As a classical
analogue, it does not make sense to talk about the center $Z(S)$
for a set $S$; one needs a few extra structures on $S$. In the
categorical setting, in order to define the Drinfeld center,
merely a plain category $C$ is not enough. Essentially, a
monoidal structure is required. Similarly, for categorical
centers of higher genera, we need essentially the braided
structures (cf \ref{def/braided-category}), which are included in
the premodular condition. Note that we will assume premodularity
for other purposes, but the categorical center of higher genera
can certainly be defined for other less restricted categories.

\subsection{Categorical center of higher genera
  ($Z$)} \label{section/categorical-center-of-higher-genera}

In this section, we formally define the categorical center of
higher genera $Z_{\sigma}(C)$ for a premodular category $C$ and an
admissible gluing $\sigma \in \Adm_{2n}$. Assume $C$ to be a premodular
category throughout this section.

\begin{definition}[$\sigma$-pair]\label{def/sigma-pair}
  Let $\sigma \in \Adm_{2n}$, i.e. $\sigma$ an admissible gluing. Define a
  $\sigma$-pair of $C$ to be a pair $(X,\gamma)$, where $X$ is a $C$-object
  and $\gamma$ is a set of half-braidings for $X$ (cf
  \ref{def/half-braiding}).
  $$
  \gamma = \{ \gamma_{[1]}, \gamma_{[2]}, \ldots, \gamma_{[2n]} \}
  $$
  satisfying pairwise commutative relations in
  \ref{def/comm-relation}.
\end{definition}

\begin{definition}[\texttt{(:comm)}, technical commutative
  relations]\label{def/comm-relation}
  Let $Z_{1}$ and $Z_{2}$ be objects in $C$, and $c$ be the
  braided structure of $C$ (so
  $a \otimes b \xrightarrow{c{a,b}} b \otimes a$). Given $[i]$ and
  $[j]$, there are three possible cases without loss of
  generality

  \begin{itemize}
    \item $[i]' < [i]'' < [j]'  < [j]''$
    \item $[i]' < [j]'  < [i]'' < [j]''$
    \item $[i]' < [j]'  < [j]'' < [i]''$
  \end{itemize}

  \begin{center}
    \includegraphics[height=5cm]{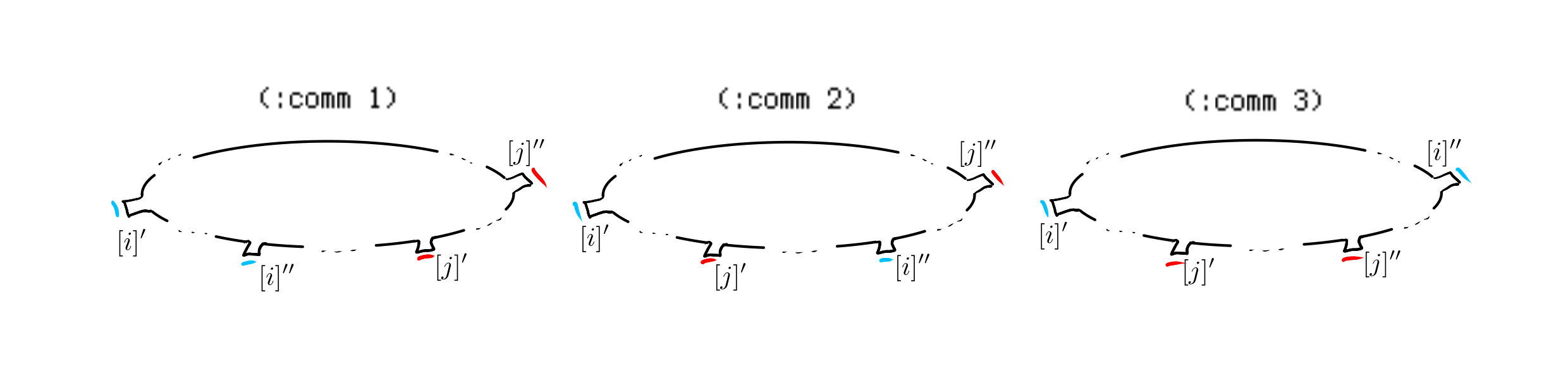}
  \end{center}

  We will give the technical conditions that the $\gamma$'s
  should obey, following by their graphical versions.

  \noindent \textbf{(1)} In the first case, $\gamma_{{[i]}}$ and
  $\gamma_{[j]}$ are required to satisfy the following commutative
  relation \texttt{(:comm 1)}, functorial in $Z_{1}$ and $Z_{2}$.

  \begin{align} \label{def/algebraic-comm-1}
    &\,(\gamma_{[j], Z_{2}} \otimes 1) (1 \otimes (c_{Z_{1},X} c_{X,Z_{1}} \gamma_{[i], Z_{1}})) \\
    =
    &\,(1 \otimes c_{Z_{1}, Z_{2}}) ((c_{Z_{1},X} c_{X,Z_{1}} \gamma_{[i], Z_{1}}) \otimes 1) (1 \otimes \gamma_{[j], Z_{2}}) (c_{Z_{1},Z_{2}}^{(-1)} \otimes 1)
  \end{align}

  \noindent \textbf{(2)} In the second case, $\gamma_{{[i]}}$ and
  $\gamma_{[j]}$ are required to satisfy the following commutative
  relation \texttt{(:comm 2)}, functorial in $Z_{1}$ and $Z_{2}$.

  \begin{align} \label{def/algebraic-comm-2}
    &\,(\gamma_{[j], Z_{2}} \otimes 1) (1 \otimes \gamma_{[i], Z_{1}}) \\
    =
    &\,(1 \otimes c_{Z_{2}, Z_{1}}^{(-1)}) (\gamma_{[i],Z_{1}} \otimes 1) (1 \otimes \gamma_{[j], Z_{2}}) (c_{Z_{1},Z_{2}}^{(-1)} \otimes 1)
  \end{align}

  \noindent \textbf{(3)} In the third case, $\gamma_{{[i]}}$ and
  $\gamma_{[j]}$ are required to satisfy the following commutative
  relation \texttt{(:comm 3)}, functorial in $Z_{1}$ and $Z_{2}$.

  \begin{align} \label{def/algebraic-comm-3}
    &\,(\gamma_{[j], Z_{2}} \otimes 1) (1 \otimes \gamma_{[i], Z_{1}}) \\
    =
    &\,(1 \otimes c_{Z_{1}, Z_{2}}) (\gamma_{[i],Z_{1}} \otimes 1) (1 \otimes \gamma_{[j], Z_{2}}) (c_{Z_{1},Z_{2}}^{(-1)} \otimes 1)
  \end{align}

  \noindent Notice that the first and the third case are almost
  the same, which is not surprising given their topological
  meaning. To make them look alike,
  define $$\tilde{\gamma}_{[i],\mbox{-}} = c_{\mbox{-},X}c_{X,\mbox{-}}\gamma_{[i],\mbox{-}}.$$

  \begin{center}
    \includegraphics[height=9cm]{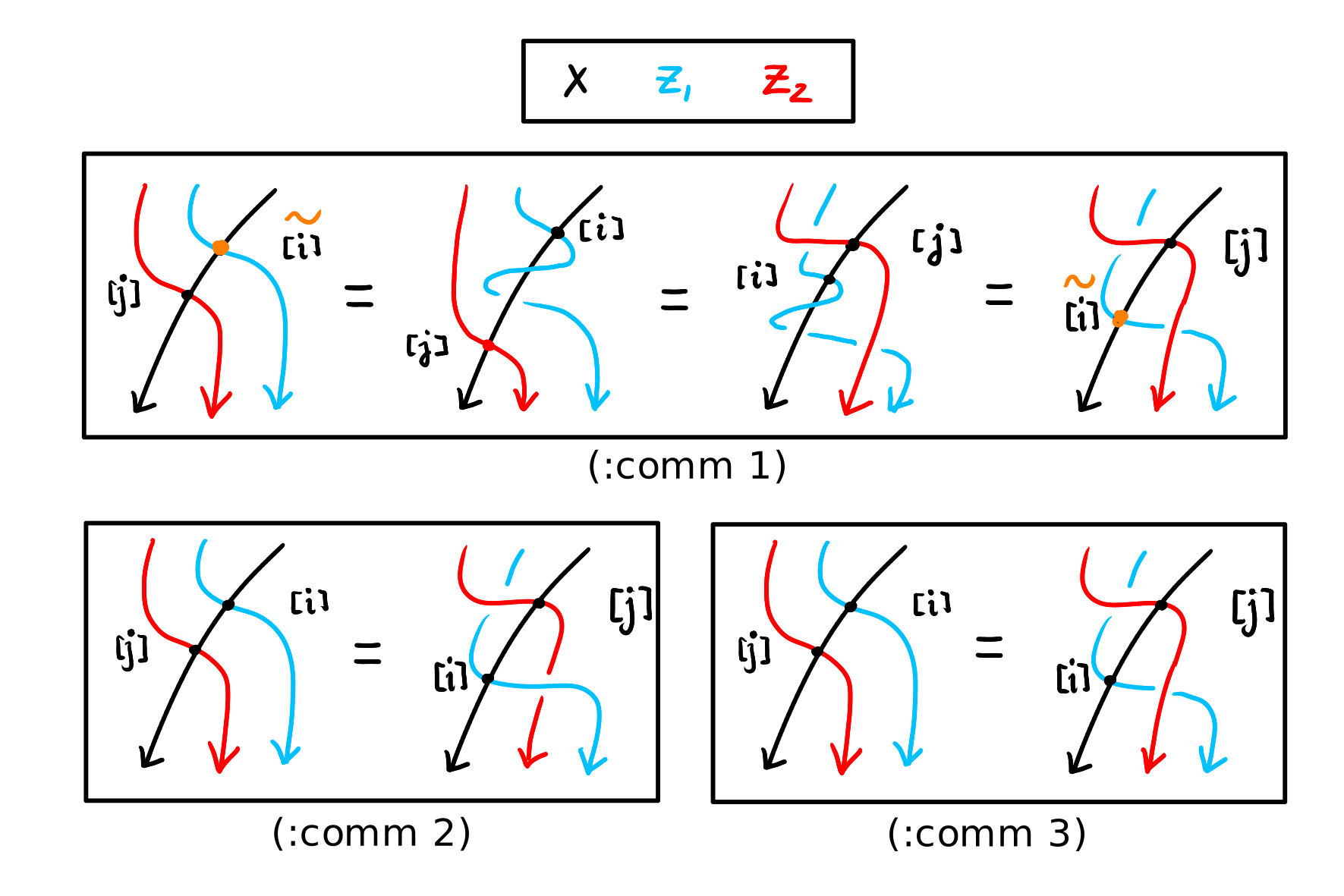}
  \end{center}
\end{definition}

\begin{definition}[$\sigma$-morphism]\label{def/sigma-morphism}
  Given two $\sigma$-pairs $\bar{X} = (X, \gamma)$ and $\bar{Y} = (Y, \beta)$
  of $C$, define $[\bar{X},\bar{Y}]_{\sigma}$ to be the linear
  subspace of $Hom_{C}(X,Y)$ consisting of the morphisms
  $(X \xrightarrow{f} Y)$ compatible with all the half-braidings
  in the following sense. For any $Z \in C$, we have (functorially
  in $Z$)

  \begin{equation} \label{def/sigma-morphism-relation}
    \beta_{Z} (1 \otimes f) = (f \otimes 1) \gamma_{Z}.
  \end{equation}

  \begin{center}
    \includegraphics[height=5cm]{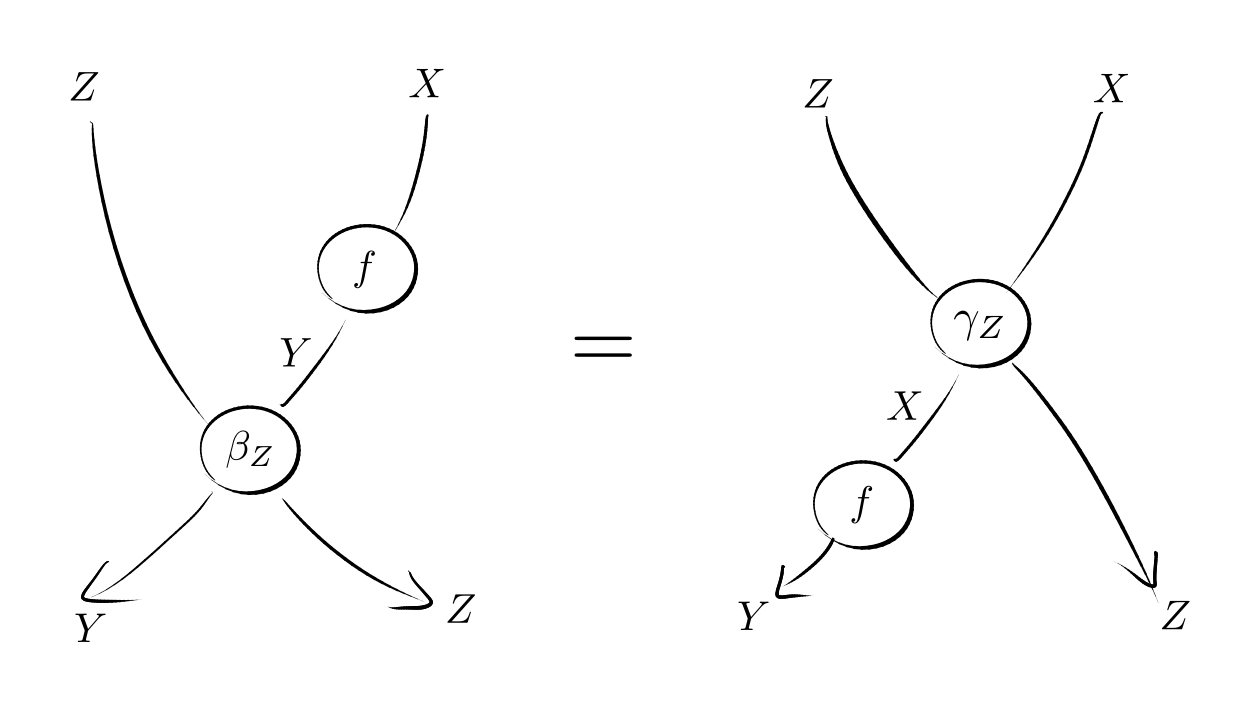}
  \end{center}

  \noindent Finally, let the identity maps and the compositions
  be inherited from that of $C$.
\end{definition}

\begin{definition}[categorical center of higher genera]\label{def/categorical-center-of-higher-genera}
  Let $C$ be a premodular category, and $\sigma \in \Adm_{2n}$ an
  admissible gluing. The categorical center of higher genera
  $Z_{\sigma}(C)$ of $C$ is defined to be the category with objects
  the $\sigma$-pairs of $C$ and with morphisms the $\sigma$-morphisms.
\end{definition}

\subsection{Properties of $Z$} \label{subsection/properties-of-Z}

In this section, we establish some basic properties of
categorical centers of higher genera. In particular, we show that
they are finite semisimple abelian categories, and that there is
a strictly ambidextrous adjunction between it and the underlying
premodular category $C$.

\subsubsection{Connecting functors}

In this subsection, we establish the relation between $C$ and its
categorical center of higher genera $Z_{\sigma}(C)$, where $C$ is a
premodular category and $\sigma$ is an admissible gluing. More
precisely, there exist two additive functors $I_{\sigma}$ and $F_{\sigma}$.
$$ I_{\sigma} : C \rightleftarrows Z_{\sigma}(C) : F_{\sigma}. $$

We will see that $I_{\sigma}$ is both a right and a left adjoints of
$F_{\sigma}$ (thus vice versa) in section
\ref{subsection/F-and-I-are-ambidextrous-adjoint}. Such a pair of
adjunction is called a (strictly) ambidextrous adjunction in the
literature.

\begin{definition}[forgetful functor]\label{def/sigma-forgetful-functor}
  The forgetful functor
  $$C \xleftarrow{F_{\sigma}} Z_{\sigma}C$$
  is defined to send objects $(X,\gamma)$ to $X$, and to send
  morphisms by inclusion (recall that the morphism space of
  $Z_{\sigma}(C)$ is defined as a subspace of that of $C$).
  Clearly, it is an additive functor.
\end{definition}

\begin{definition}[induction functor]\label{def/sigma-induction-functor}
  The induction functor
  $$C \xrightarrow{I_{\sigma}} Z_{\sigma}(C)$$
  is more complicated, so will be defined step-by-step. Define
  $I_{\sigma}(X)$ to be $(X_{\sigma}, \gamma)$, which is given below.

  \noindent Let $\mathcal{O}(C)$ be the set of isomorphism
  classes of simple objects of $C$, and $o(C)$ a set of
  representatives. To each $C$-object $X$, define another
  $C$-object
  \begin{equation}\label{def/x-sigma}
    X_{\sigma} := \bigoplus_{o(C)^{n}} (X_{1} \otimes X_{2} \otimes \ldots \otimes X_{n} \otimes X \otimes X_{n+1} \otimes \ldots \otimes X_{2n-1} \otimes X_{2n}),
  \end{equation}
  where $X_{k}$ runs through $o(C)$ if $k = [k]'$ or
  $X_{k} = X_{([k]')}^{\star}$ if $k = [k]''$.
  Here, recall that the first case means that $k$ is the smaller
  member in the set $[k]$, while the second case means that $k$
  is the larger member. For example,
  $$X_{(1 2)} = \oplus_{X \in o(C)} X \boxtimes X^{\star}.$$
  Notice that it does not depend on the choice of $o(C)$ up to
  canonical isomorphisms. Similarly, neither does $X_{\sigma}$ for
  general admissible gluings $\sigma$.

  Next, to each $k$, we define the $[k]$-th half-braiding
  $\gamma_{[k]}$ for $X_{\sigma}$
  \begin{equation} \label{def/induced-half-braiding}
    (-) \otimes X_{\sigma} \xrightarrow{\gamma_{[k]}} X_{\sigma} \otimes (-)
  \end{equation}
  as in the following picture.
  \begin{center}
    \includegraphics[height=5cm]{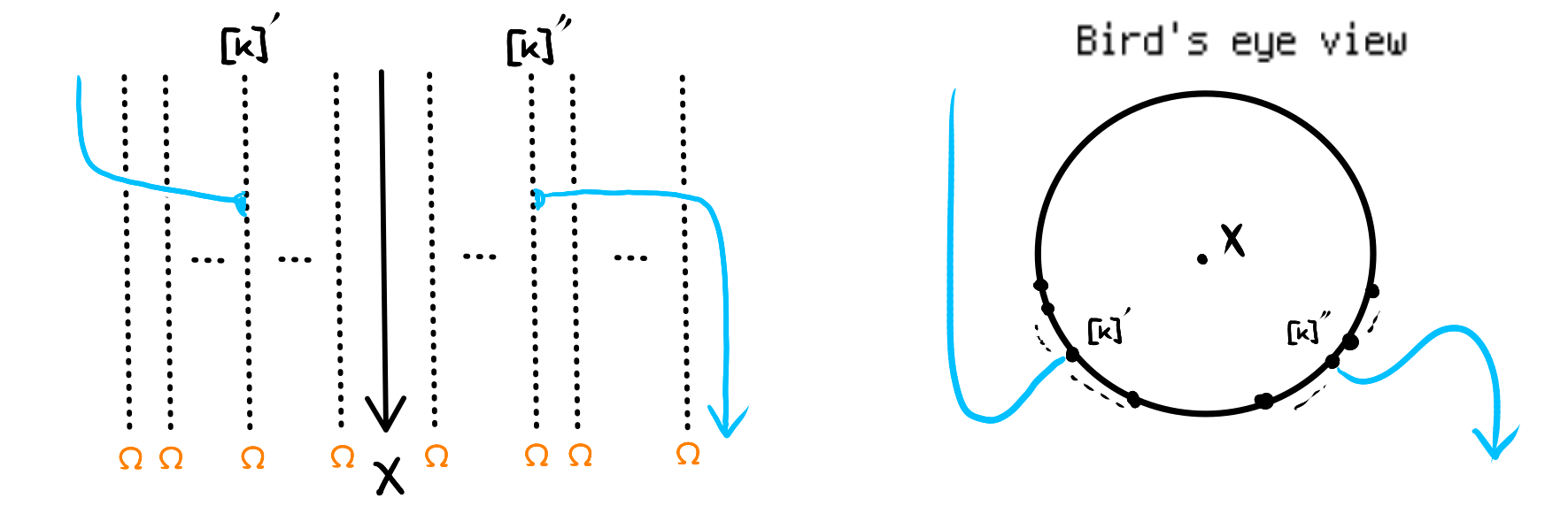}
  \end{center}
  where the nontrivial intertwiners are given precisely in
  \ref{lemma/fundamental-lemma-of-Omega}.

  We contend that so given $(X_{\sigma}, \gamma)$ is indeed an object of
  $Z_{\sigma}(C)$. One needs to show that each element in $\gamma$ is a
  half-braiding, and that $\gamma$ satisfies the commutative relations
  (\ref{def/comm-relation}). A proof of this can be found in
  (\ref{lemma/induced-half-braiding-satisfies-comm-relations}).
  It remains to define the map on the morphism spaces
  $Hom_{C}(X,Y)$. Given a morphism $X \xrightarrow{f} Y$, one
  defines
  \begin{equation}\label{def/induced-arrow}
    I_{\sigma}(f) := \overline{f} := \id^{\otimes n} \otimes f \otimes \id^{\otimes n}.
  \end{equation}

  \noindent To conclude, it remains to show that
  \begin{itemize}
    \item The morphism $\overline{f}$ is compatible with the
          half-braidings $\gamma$ and $\beta$.
    \item The construction $(\overline{-})$ preserves the
          identities and the compositions.
  \end{itemize}

  \noindent The first point is shown in
  (\ref{lemma/induced-arrow-respect-half-braidings}). The second
  point is clear.
\end{definition}

\subsubsection{Ambidextrous adjunction} \label{subsection/F-and-I-are-ambidextrous-adjoint}

\noindent Both functors $F_{\sigma}$ and $I_{\sigma}$ are additive
immediately by definition. We are ready to state and prove the
main statement of this subsection.

\begin{theorem}\label{theorem/F-and-I-are-ambidextrous-adjoint}
  The functors
  $$ I_{\sigma} : C \rightleftarrows Z_{\sigma}(C) : F_{\sigma} $$
  so defined in \ref{def/sigma-forgetful-functor} and
  \ref{def/sigma-induction-functor} are (strictly) ambidextrous
  adjoint to each other. In other words, $I_{\sigma}$ is both a
  left adjoint and a right adjoint of $F_{\sigma}$, thus vice
  versa.
\end{theorem}
\begin{Proof}
  We will prove that $F_{\sigma}$ is right adjoint to $I_{\sigma}$. Namely,
  we need to show that for each $C$-object $X$ and for each
  $Z_{\sigma}(C)$-object $(Y,\beta)$, there is a vector space isomorphism
  $$F: Hom_{C}(X,F_{\sigma}(Y,\beta)) \rightleftarrows Hom_{Z_{\sigma}(C)}(I_{\sigma}(X), (Y,\beta)): G.$$
  It will then be obvious that the other side can be proved
  verbatim by taking duals (or by flipping the graph, in terms of
  graphical calculus). To prove such equivalence, we construct
  explicit maps for both sides, and argue that each composition
  equals the identity map.

  Given $\phi \in Hom_{C}(X,F_{\sigma}(Y,\beta))$, define its image on the
  other side to be
  $$F(\phi) := \frac{1}{dim(\Omega)^{n}}
  (\widehat{\Pi}_{k \in \{[1], \ldots , [2n]\}} \beta_{k, \Omega}^{\star}) \circ (1 \otimes \ldots \otimes 1 \otimes \phi \otimes 1 \otimes \ldots \otimes 1),$$
  where $\Omega$ is the shorthand notation given in \ref{def/Omega},
  and the term
  $\widehat{\Pi}$ is explained below: The term $\widehat{\Pi}$ is a
  $C$-morphism $I_{\sigma}(Y) \to Y$. Each $\beta_{[i], \Omega}^{\star}$ is a
  $C$-morphism $\Omega \otimes Y \otimes \Omega \to Y$, induced from
  $\Omega \otimes Y \xrightarrow{\beta_{[i], \Omega}} Y \otimes \Omega$ by composing the
  evaluation map (note that $\Omega^{\star} \simeq \Omega$). So the
  $\beta_{[i],\Omega}^{\star}$'s are maps that kills the $[i]'$-th and the
  $[i]''$-th component of $\Omega$ by using $\beta_{[i]}$. However,
  depending on the combinatorial nature of $\sigma \in \Adm_{2n}$, one
  should insert suitable braidings for it to make sense. For
  example, if $n = 3$, $[1] = [4]$, $[2] = [6]$, and $[3] = [5]$,
  we define the $\widehat{\Pi}$ term as in the following diagram --
  the order of the $[i]$'s does not really matter, thanks to
  \ref{def/comm-relation}.

  \begin{center}
    \includegraphics[height=4cm]{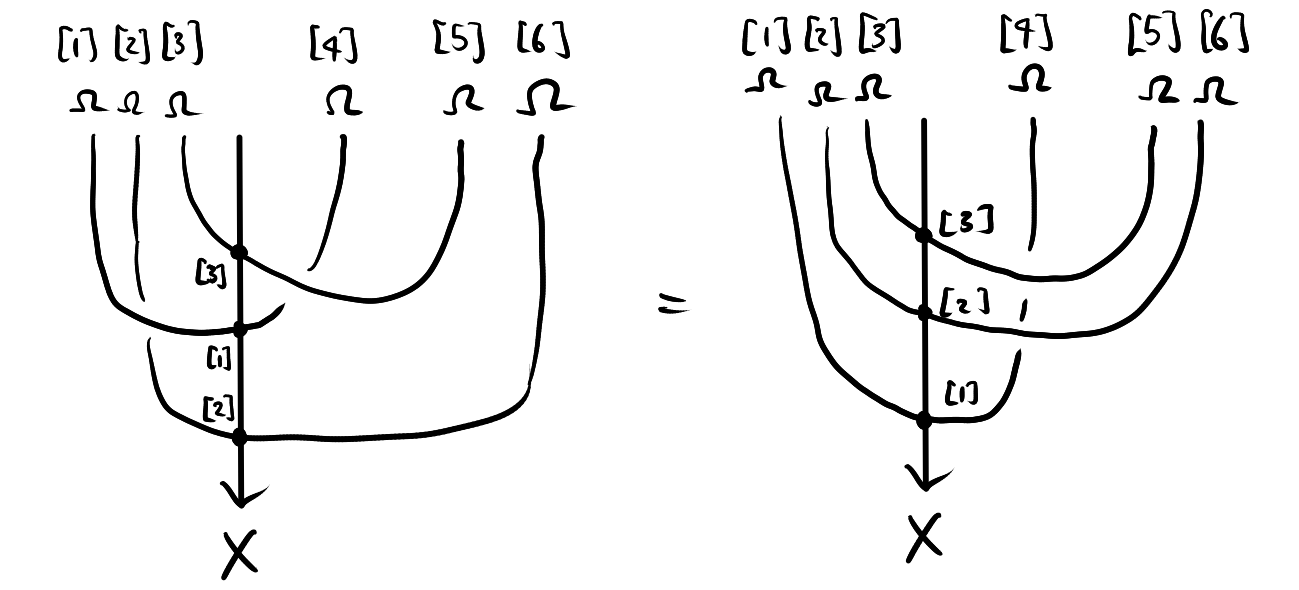}
  \end{center}

  That $F(\phi)$ is indeed a morphism in $Z_{\sigma}(C)$ follows directly
  from the commutative relation \ref{def/comm-relation}, that
  half-braidings are by definition monoidal, and the sliding
  lemma \ref{lemma/sliding-lemma}.

  \begin{center}
    \includegraphics[height=3cm]{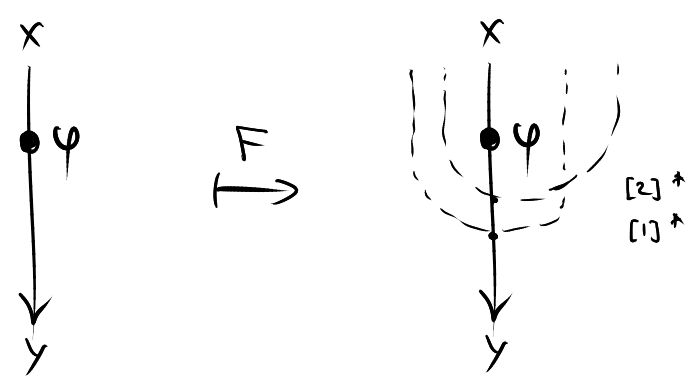}
  \end{center}

  On the other hand, given $\psi \in Hom_{Z_{\sigma}(C)}(I_{\sigma}(X), (Y,\beta))$,
  define its image $G(\psi)$ on the other side to be as indicated in
  the graph below.

  \begin{center}
    \includegraphics[height=3cm]{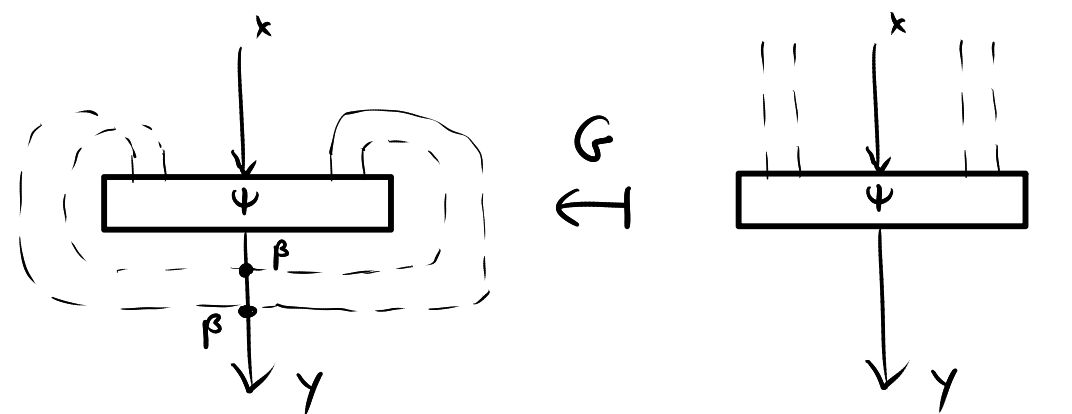}
  \end{center}

  To prove that $GF$ is the identity map, use the fact that the
  half-braidings are by definition functorial. Hence one can
  slide the $\Omega$'s out the axis. Finally, the product of the
  dimensions of the $\Omega$'s cancel with the denominator.

  \begin{center}
    \includegraphics[height=7cm]{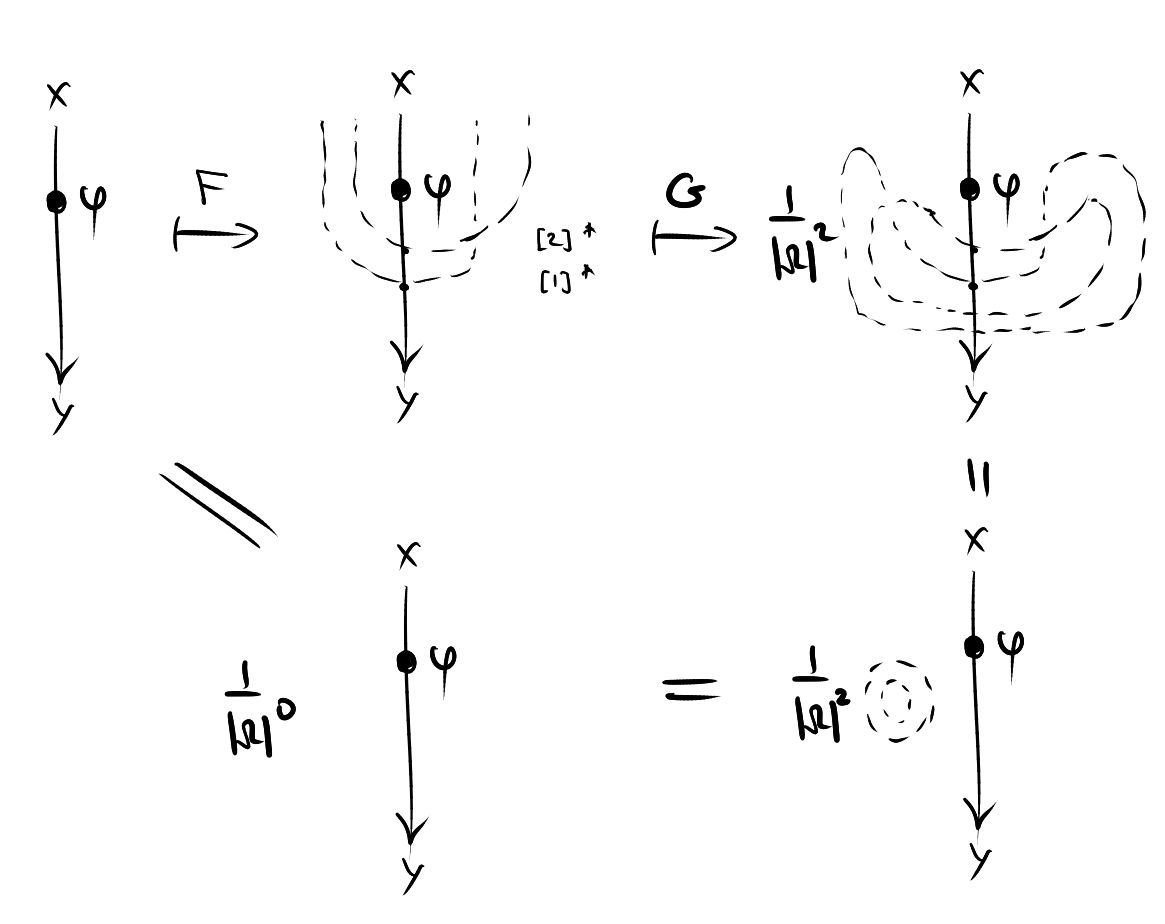}
  \end{center}

  \noindent To prove that $FG$ is the identity map, use the
  sliding lemma again. Then use the assumption that $\psi$ is a
  $Z_{\sigma}(C)$-morphism to drag $\Omega$ down. Finally, slide
  the $\Omega$'s away from the axis as in the case for $GF=1$.

  \begin{center}
    \includegraphics[height=11cm]{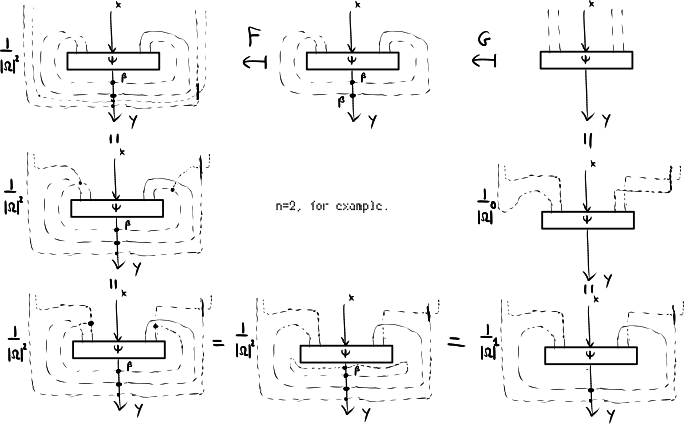}
  \end{center}

\end{Proof}

\subsubsection{$Z$ is finite semisimple abelian}

In this section, we show that the categorical centers of higher
general over premodular categories are finite semisimple abelian
categories.

\begin{lemma}[(monadic) projection]\label{corollary/monadic-projection}
  Let $\sigma \in Adm_{2n}$ be an admissible gluing, $C$ be a premodular
  category, $Z = Z_{\sigma}(C)$ be the categorical center of higher
  genera of $C$ with respect to $\sigma$, and $\overline{X}=(X,\gamma_{\cdot})$
  and $\overline{Y}=(Y,\beta_{\cdot})$ be $Z$-objects.

  Recall that the morphism space $Z(\overline{X},\overline{Y})$
  is by definition a subspace of $C(X,Y)$. Then there is a
  natural projection $\pi_{\gamma_{\cdot},\beta_{\cdot}} \in End(Mor_{C}(X,Y))$ to the
  subspace $Z(\overline{X},\overline{Y})$ that respects the
  composition.
\end{lemma}
\begin{Proof}
  The full proof is tedious and postponed to
  \ref{section/adjunctions-as-monads}. In particular, see
  \ref{example/main-example-of-monadic-projection}. Roughly, the
  statement follows from the monadic nature of the strictly
  ambidextrous adjunctions and a condition called the ``unity
  trace condition'' (\ref{def/unity-trace}).
\end{Proof}

\begin{theorem} \label{thm/cat-prop-of-Z}
  Let $C$ be a premodular category, $\sigma \in \Adm_{2n}$ an
  admissible gluing. Then the categorical center of higher genera
  $Z_{\sigma}(C)$ is a finite semisimple abelian category.
\end{theorem}
\begin{Proof}

  The complete proof is tedious and thus postponed to the
  appendix. See \ref{lemma/higher-cat-center-is-abelian},
  \ref{lemma/higher-cat-center-is-semisimple}, and
  \ref{lemma/higher-cat-center-is-finite}. The main idea is to
  make heavy use of the projection
  \ref{corollary/monadic-projection}.
\end{Proof}

\section{Proof of the main
  statement} \label{section/proof-for-the-main-statement}

In this section, we prove the main statement of this paper.

\begin{theorem}[Main Statement] \label{main-statement} Let $C$ be
  a premodular category, $\sigma \in \Adm_{2n}$ an admissible
  gluing, $\Sigma = \Sigma_{\sigma}$ the surface constructed from
  $\sigma$. Then the Crane-Yetter theory of $\Sigma_{\sigma}$
  over $C$ and the categorical center of higher genera
  $Z_{\sigma}(C)$ are equivalent as $\mathbb{k}$-linear
  categories
  $$CY_{\Sigma_{\sigma}}(C) \simeq Z_{\sigma}(C).$$

  \noindent As the $Z_{\sigma}(C)$'s are proven to be finite
  semisimple abelian \ref{thm/cat-prop-of-Z}, the Crane-Yetter
  theory for each open surface is also a finite semisimple
  abelian category.
\end{theorem}

\noindent To stress the informal aspect again, we recall
\ref{remark/shape-algebra-duality}.

\begin{remark}\label{recall/remark/shape-algebra-duality}
  In $H_{1}(S^{1}) \simeq \mathbb{Z}$, one sees the algebra of
  the shape $S^{1}$ and the shape of the algebra $\mathbb{Z}$.
  Our main result should be viewed as a higher analogue. That is,
  one sees the (higher) algebra of the shape $\Sigma_{\sigma}$
  and the shape of the (higher) algebra $Z_{\sigma}$.
\end{remark}

\begin{example}[$n=0$]
  For $n=0$, the surface $\Sigma$ is the open disk, the categorical
  center of higher genera reduces to the underlying premodular
  category $C$, so the theorem recovers that $CY_{\Sigma}(C) \simeq C$ as
  shown in \ref{thm/cy-of-a-disk}.
\end{example}
\begin{example}[$n=1$]
  For $n=1$, the only possible surface is the cylinder, the
  categorical center of higher genera reduces to the Drinfeld
  center $Z(C)$. Hence, the theorem recovers that
  $CY_{\Sigma}(C) \simeq Z(C)$ as shown in
  \ref{thm/drinfeld-center-and-cylinder}.
\end{example}
\begin{example}[$n=2$]
  For $n=2$, there are
  two possible surfaces: the $1$-punctured torus and the
  $3$-punctured disk. In the former case, the categorical center
  of higher genera reduces to the elliptic center $Z^{el}(C)$, so
  the theorem recovers that $CY_{\Sigma}(C) \simeq Z^{el}(C)$ as shown in
  \ref{thm/elliptic-center-and-punctured-torus}. In the later
  case, the theorem provides a new result.
\end{example}

\subsection{Strategy}

\[\begin{tikzcd}
	{CY_{\Sigma_{\sigma}}(C)} && {Z_{\sigma}(C)} \\
	{\Kar(cy_{\Sigma_{\sigma}}(C))} & {\Kar(ho.cy_{\Sigma_{\sigma}}(C))} \\
	{cy_{\Sigma_{\sigma}}(C)} & {ho.cy_{\Sigma_{\sigma}}(C)}
	\arrow[from=3-1, to=2-1, hook]
	\arrow["{:=}", from=1-1, to=2-1]
	\arrow[from=3-2, to=2-2, hook]
	\arrow["{\simeq}", from=3-2, to=3-1, hook']
	\arrow["{J}"{name=0, swap}, from=2-2, to=1-3]
	\arrow["{G}"{name=1, swap}, from=1-3, to=2-2, curve={height=12pt}]
	\arrow["{j}"', from=3-2, to=1-3, curve={height=12pt}]
	\arrow["{\simeq}", from=2-1, to=2-2, no head]
	\arrow[Rightarrow, "{\simeq}" description, from=0, to=1, shorten <=1pt, shorten >=1pt, phantom, no head]
\end{tikzcd}\]

\begin{enumerate}
  \item (Condensation of string nets) By definition,
        $CY_{\Sigma_{\sigma}}(C)$ is the Karoubi envelope of
        $cy_{\Sigma_{\sigma}}(C)$. Find an equivalent subcategory
        $ho.cy_{\Sigma_{\sigma}}(C)$ of $cy_{\Sigma_{\sigma}}$ by
        reducing topological data. Then of
        course $$\Kar(cy_{\Sigma_{\sigma}}) \simeq \Kar(ho.cy_{\Sigma_{\sigma}}(C)).$$
  \item (top $\to$ alg) Construct a
        functor $$ho.cy_{\Sigma_{\sigma}}(C) \xrightarrow{j} Z_{\sigma}(C),$$
        and extend it
        to $$\Kar(ho.cy_{\Sigma_{\sigma}}(C)) \xrightarrow{J} Z_{\sigma}(C).$$
  \item (top $\leftarrow$ alg) Construct a
        functor $$\Kar(ho.cy_{\Sigma_{\sigma}}(C)) \xleftarrow{G} Z_{\sigma}(C).$$
  \item Argue that the compositions $J \circ G$ and $G \circ J$ are
        equivalent to the identity functors.
  \item Show that the equivalence is of finite semisimple abelian
        categories.
\end{enumerate}

\subsection{Proof}

In this section, we give the full proof of the main statement.
Each subsection corresponds to each step in the outlined
strategy.

\subsubsection{Reducing topological data}

In this subsection, we reduce the topological data by
constructing a smaller yet equivalent subcategory
\begin{equation}
ho.cy_{\Sigma_{\sigma}}(C) \xrightarrow[\sim]{\subset} cy_{\Sigma_{\sigma}}(C).
\end{equation}

\begin{definition}[$ho.cy_{\Sigma_{\sigma}}(C)$]
  Let $C$ be a premodular category and $\sigma$ an admissible gluing.
  The subcategory $ho.cy_{\Sigma_{\sigma}}(C)$ of $cy_{\Sigma_{\sigma}}(C)$ is
  defined as follows.

  Let $p$ be the central point of the standard disk. An object of
  $ho.cy_{\Sigma_{\sigma}}(C)$ is defined to be the single $C$-colored
  point $(p, X)$ for some $X \in \Obj(C)$. A morphism from $(p, X)$
  to $(p, Y)$ is the equivalence class in which the following
  string net lives.

  \begin{center}
  \includegraphics[height=5cm]{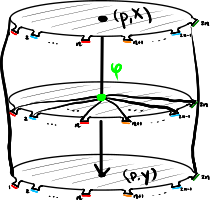}
  \end{center}

  \noindent Clearly, $ho.cy_{\sigma}(C)$ is a subcategory of $cy_{\sigma}(C).$
\end{definition}

\begin{theorem}[equivalence of reduction]
  The inclusion
  functor $$\iota: ho.cy_{\Sigma_{\sigma}}(C) \xrightarrow{\subset} cy_{\Sigma_{\sigma}}(C)$$
  is an equivalence of categories. Clearly, it is additive.
\end{theorem}
\begin{Proof}
  By a basic lemma in category theory, it is enough to show that
  $\iota$ is fully faithful and essentially
  surjective. \\

  \noindent (Essentially surjective) Recall that a typical object
  of $cy_{\Sigma_{\sigma}}(C)$ is a finite collection of $C$-colored points
  on $\Sigma_{\sigma}$. It suffices to find an equivalent object of the
  form $(p, X)$, for some $X \in \Obj(C)$. This can be done by the
  following reductions.

  \begin{enumerate}
    \item Slightly push the points on the boundary into the
          smaller side.
    \item Compress the points from the legs into the disk.
    \item Then compress further for the points to stay in a small
          unit disk in the middle.
    \item Project the objects to a fixed line.
    \item Take their tensor products.
  \end{enumerate}

  \noindent Each step above can be realized as an isomorphism in
  $CY_{\Sigma_{\sigma}}(C)$, so every object is isomorphic to an object in
  $ho.cy_{\Sigma_{\sigma}}(C)$.

  \begin{center}
  \includegraphics[width=15cm]{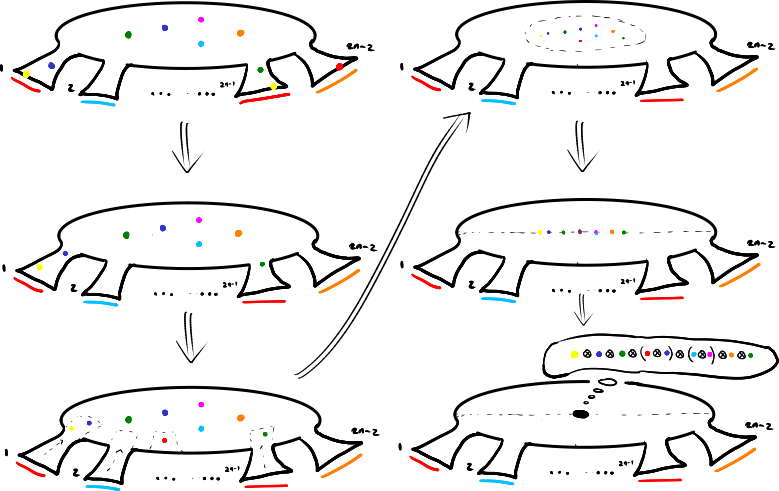}
  \end{center}

  \noindent (Fully faithful) We ought to show that
  $$
  Hom_{ho.cy_{\Sigma_{\sigma}}(C)}((p,X), (p,Y)) \xrightarrow{\iota} Hom_{cy_{\Sigma_{\sigma}}(C)}((p,X), (p,Y))
  $$
  is an equivalence of vector spaces. Clearly, it is linear and
  injective, as the quotient relations on both sides are the
  same. To prove surjectivity, we have to show that any arbitrary
  string net with given boundary condition is equivalent to a
  stringnet given in the definition of $ho.cy_{\Sigma_{\sigma}}(C)$. This
  can be done by a similar compression process as in the proof of
  essential surjectivity.

  \begin{enumerate}
    \item Push the stringnets away from the end of the legs.
    \item Push the stringnets away from the legs.
    \item Compress everything into a fixed central bar.
    \item Replace the stringnets through boundaries with one
          strand for each leg by taking tensor products by using
          the Reshetikhin-Turaev evaluation.
    \item Then compress vertically.
    \item Then finally replace replace the tangled mess in the
          middle by a morphism.
  \end{enumerate}

  \begin{center}
  \includegraphics[width=15cm]{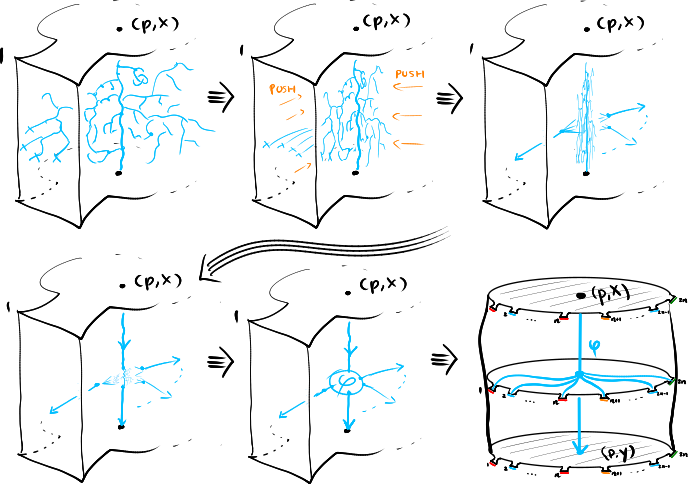}
  \end{center}
\end{Proof}

\subsubsection{topology $\rightarrow$ algebra}

In this subsubsection, we aim to construct a morphism
$$\Kar(ho.cy_{\Sigma_{\sigma}}(C)) \xrightarrow{J} Z_{\sigma}(C).$$

\noindent As $Z_{\sigma}(C)$ is abelian (\ref{thm/cat-prop-of-Z}), by
\ref{lemma/additive-functor-to-abelian-cat-lifts-to-karoubi} we
only have to construct an additive functor
$$
ho.cy_{\Sigma_{\sigma}}(C) \xrightarrow{j} Z_{\sigma}(C).
$$

To define $j$, recall that a typical object of $ho.cy_{\Sigma_{\sigma}}(C)$
is a colored point $(p, X)$, where $p$ denotes the central point
of the standard disk. Define its image under $j$ to be $X_{\sigma}$ as
in (\ref{def/x-sigma}). A typical morphism from $(p, X)$ to
$(p, Y)$ is a linear combination of the equivalence classes of
the stringnets like $\Gamma$. Define the image of $[\Gamma]$ under $j$ to
be
\begin{center}
  $[\Gamma]$ := \includegraphics[height=5cm]{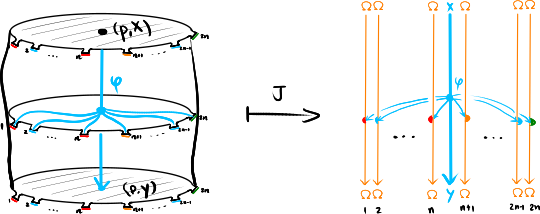}
\end{center}
where $\Omega$ is the shorthand notation given in \ref{def/Omega}
crossings mean the
braidings of $C$, and the nontrivial pairs of intertwiners are
given in \cite[(1.8)]{kirillov-balsam/turaev-viro-I}. Extend the
definition additively, and then we have our desired additive
functor $j$.

\subsubsection{topology $\leftarrow$ algebra}

In this subsubsection, we construct a functor
$$\Kar(ho.cy_{\Sigma_{\sigma}}(C)) \xleftarrow{G} Z_{\sigma}(C).$$
Recall that a typical object in $Z_{\sigma}(C)$ is $(X, \gamma)$,
where $X \in \Obj(C)$ and $\gamma$ is a set of half-braidings

$$\gamma = \{ \gamma_{[1]}, \gamma_{[2]}, \ldots \gamma_{[2n]}\}$$

\noindent satisfying some relations \ref{def/half-braiding}.
Define the image of $(X,\gamma)$ under $G$ to be $((p, X), \pi_{\gamma})$,
where $p$ denotes the central point of the standard disk, and
$\pi_{\gamma}$ to be the equivalence class of the following stringnets.

\begin{center}
$\frac{1}{|\Omega|^{n}}$\includegraphics[width=7cm]{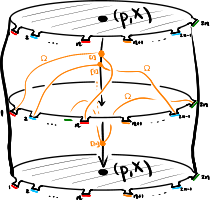}
\end{center}

\noindent That $\pi_{\gamma}$ is a projection follows from the
commutative relations (\ref{def/comm-relation}) and a graphical
lemma \cite[(3.7)]{kirillov/string-net}. For morphisms, define
the image of $(X, \gamma) \xrightarrow{f} (Y, \beta)$ under $G$ to be

\begin{center}
\includegraphics[width=7cm]{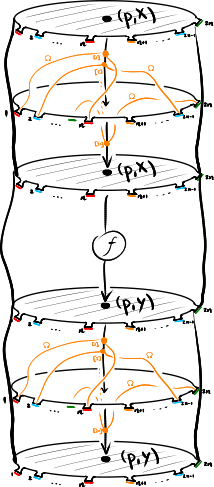}
\end{center}

\noindent which is indeed a morphism in the Karoubi envelope
because $\pi_{\gamma}$ and $\pi_{\beta}$ are idempotents.

\subsubsection{topology $\leftrightarrow$ algebra}

In this subsection, we will show that $J \circ G$ and $G \circ J$ are
equivalent to identity functors.

That $G \circ J \simeq 1$ follows directly from the same argument of
\cite[Figure 15]{kirillov/string-net}; we just have to do it $n$
times. On the other hand, in fact we have $J \circ G = 1$. Indeed,
denote $(J \circ G)((X,\gamma)) = (X',\gamma')$. That $X' = X$ follows directly
from the sliding lemma \ref{lemma/sliding-lemma},
and that $\gamma' = \gamma$ follows
from the sliding lemma, and the fact that half-braidings are by
definition monoidal.

Finally, since $G$ and $J$ are additive, this proves the
equivalence of both sides as abelian categories. Therefore,
$CY_{\Sigma_{\sigma}}(C)$ and $Z_{\sigma}(C)$ are equivalent as finite
semisimple abelian categories.

\section{Outlook and remarks} \label{section/outlook-and-remarks}

In this section, we describe some open directions and more work
in progress.

\vspace{5mm}

\noindent\textbf{Surface combinatorics}

Let $\Sigma$ be an open surface of a fixed topological type. In
general, there are many different admissible gluings $\sigma$ with
$\Sigma_{\sigma} \simeq \Sigma$. As Crane-Yetter theory is topological, we have many
differently-presented categories that are in fact equivalent.

Moreover, the excision property
\ref{thm/excision-principle}
\[
CY_{\Sigma_{1} \cup \Sigma_{2}}(C) \simeq CY_{\Sigma_{1}}(C) \boxtimes_{CY_{{\Sigma_{1} \cap \Sigma_{2}}}(C)} CY_{\Sigma_{2}}(C)
\]
provides more ways to obtain $\Sigma$. It is an interesting
work to establish explicit equivalences.

\vspace{5mm}

\noindent\textbf{Surfaces without punctures}

The main statement of this work provides a nice description for
the Crane-Yetter theory of any surface with at least one
puncture. While the case without punctures can be taken care
easily by patching with the excision principle, the resulting
categories are described in terms of balanced (Deligne) tensor
products, which are more obscure. The author believes that there
should be a better description.

\vspace{5mm}

\noindent\textbf{Module categorical structures}

The Drinfeld center with the stacking tensor product
$\overline{\otimes}$ acts on $CY_{\Sigma}(C)$ in possibly multiple ways. We
will establish the module categorical structure for $CY_{\Sigma}(C)$
explicitly in future work.

\vspace{5mm}

\noindent\textbf{Concrete computations}

Compute examples for Crane-Yetter theory in dimension two and
three explicitly and concretely, especially for premodular
categories $C$ that are neither modular nor symmetric. There are
a few candidates. The first is the even part of the
semi-simplification of $\Rep(U_{q}\mathfrak{sl}_{2})$ for special
$q$. Another family of examples are given by
$\Rep(\mathfrak{X})$, where $\mathfrak{X}$ denotes a finite
$2$-group \cite{bantay/chars-xmods}. Compute C-Y invariants for
3- and 4- folds directly from our result, and seek for insights.

\vspace{5mm}

\noindent\textbf{Minimal data for Crane-Yetter}

When $C$ is modular, $CY$ in dimension two trivializes to the
number of punctures. In particular, for closed surfaces $\Sigma$,
$CY_{\Sigma}(C)$ trivialize to the Muger center of $C$, which is just
$(Vect)$ due to the modularity \cite[Prop
8.20.12]{egno/tensor-cats}. On the other hand, when $C$ is not
modular, $CY_{\Sigma}(C)$ does no seem to depend on full information
from $C$. Find the minimal data needed in order to determine
$CY_{\Sigma}(C)$.

\vspace{5mm}

\noindent\textbf{Piecewise-linear setting}

For exposing stringnets with simplicity, we assume smooth
structures for our surfaces. Crane-Yetter theory can be made
precise in the PL setting. This is not necessary and should be
removed in future work. For an analogue in one dimension lower,
see \cite{kirillov-balsam/turaev-viro-I}.

\section{Appendix} \label{section/appendix}

Most sections in the appendix are added for the sake of
completeness.

\subsection{Abelian categories}

A complete definition of an abelian category is given in this
subsection. In particular, see \ref{def/abelian-category}.

\begin{definition}[pre-additive category]\label{def/preadditive-category} \cite[I.8. p.28]{mac-lane/cats-for-working-mathematician}
  A pre-additive category, or called an Ab-category, is a
  category $A$ in which each hom-set is an (additive) abelian
  group, with respect to which the composition maps are bilinear.
\end{definition}

\begin{definition}[biproduct]\label{def/biproduct}\cite[VIII.2. Definition]{mac-lane/cats-for-working-mathematician}
  Let $A$ be an pre-additive category
  (\ref{def/preadditive-category}). For each pair of $A$-objects
  $(a,b)$, define their biproduct to be the pair
  $(c,\{p_{a},p_{b},i_{a},i_{b}\})$, where $c$ is an $A$-object,
  $p_{a}$ and $p_{b}$ are morphisms from $c$ to $x$, $i_{a}$ and
  $i_{b}$ are morphisms from $a$ and $b$ to $c$, with the
  equations satisfied:
  \begin{align}
    1_{a} &= p_{a}i_{a} \\
    1_{b} &= p_{b}i_{b} \\
    1_{c} &= i_{a}p_{a} + i_{b}p_{b}.
  \end{align}
\end{definition}

\begin{definition}[initial object]\label{def/initial-object}\cite[p.20]{mac-lane/cats-for-working-mathematician}
  Let $C$ be a category. An initial object $s$ in $C$ is a
  $C$-object such that to each $C$-object $a$ there is exactly
  one $C$-morphism $s \to a$.
\end{definition}

\begin{definition}[terminal object]\label{def/terminal-object}\cite[p.20]{mac-lane/cats-for-working-mathematician}
  Let $C$ be a category. A terminal object $t$ in $C$ is an
  $C$-object such that to each object $a$ there is exactly one
  morphism $a \to t$.
\end{definition}

\begin{definition}[null object]\label{def/null-object}\cite[p.20]{mac-lane/cats-for-working-mathematician}
  Let $C$ be a category. A null object $z$ is a $C$-object which
  is both initial (\ref{def/initial-object}) and terminal
  (\ref{def/terminal-object}).
\end{definition}

\begin{definition}[additive category]\label{def/additive-category}\cite[VIII.2. p.196]{mac-lane/cats-for-working-mathematician}
  An additive category $A$ is an pre-additive category
  (\ref{def/preadditive-category}) that satisfies the following
  conditions
  \begin{itemize}
    \item $A$ has a null object (\ref{def/null-object}).
    \item $A$ has a binary biproduct for each pair of $A$-objects (\ref{def/biproduct}).
  \end{itemize}
\end{definition}

\begin{definition}[zero morphism] \label{def/zero-arrow}
  \cite[VIII.1.]{mac-lane/cats-for-working-mathematician} Let $C$
  be a category with a null object $z$ (\ref{def/null-object}).
  Let $a, b$ be $C$-objects. The zero morphism from $a$ to $b$ is
  defined to be the composition of the morphism from $a$ to $z$ and
  the morphism from $z$ to $b$
  $$(a \xrightarrow{0} b) := (a \to z \to b).$$
\end{definition}

\begin{definition}[monic morphism]\label{def/monic-arrow}\cite[p.19]{mac-lane/cats-for-working-mathematician}
  Let $C$ be a category. A monic morphism is a $C$-morphism
  $a \xrightarrow{m} b$ such that the left cancellation rule
  holds:
  \begin{equation}
    (mf=mg) \Rightarrow (f=g).
  \end{equation}
\end{definition}

\begin{definition}[epi morphism]\label{def/epi-arrow} \cite[p.19]{mac-lane/cats-for-working-mathematician}
  Let $C$ be a category. An epi morphism is a $C$-morphism
  $a \xrightarrow{m} b$ such that the right cancellation rule
  holds:
  \begin{equation}
    (fe=ge) \Rightarrow (f=g).
  \end{equation}
\end{definition}

\begin{definition}[diagonal functor] \label{def/diagonal-functor}
  Let $C$ and $J$ be categories. The diagonal functor $\Delta$
  from $C$ to $C^{J}$ is defined to send each $C$-object $c$ to
  the constant functor $\Delta_{c}$, and to send each $C$-morphism
  $c \xrightarrow{f} c$ to the constant natural transform
  $\Delta_{f}$. (cf
  \cite[p.67]{mac-lane/cats-for-working-mathematician}).
\end{definition}

\begin{definition}[universal morphism] \label{def/universal-arrow}
  \cite[III.1.]{mac-lane/cats-for-working-mathematician} Let $C$
  and $D$ be categories. Let $c$ be an $C$-object. Let
  $D \xrightarrow{S} C$ be a functor. A universal morphism from $c$
  to $S$ is a pair $(r,u)$, where $r$ is $D$-object and
  $c \xrightarrow{u} Sr$ is an $C$-morphism that satisfies the
  following condition:

  For each pair $(d,f) \in Obj(D) \times C(c,Sd)$, there is a
  unique $D$-morphism $r \xrightarrow{f'} d$ with $Sf' \circ u=f$.
\end{definition}

\begin{definition}[categorical
  limit] \label{def/categorical-limit}
  \cite[III.4.]{mac-lane/cats-for-working-mathematician} Let $C$
  and $J$ be categories and $J \xrightarrow{F} C$ be a functor. A
  limit for the functor $F$ is defined to be a universal morphism
  (\ref{def/universal-arrow}) $(r,v)$ from $\Delta$ to $F$, where
  $\Delta$ is the diagonal functor (\ref{def/diagonal-functor})
  from $C$ to $C^{J}$.
\end{definition}

\begin{definition}[equalizer]\label{def/equalizer}\cite[III.4.]{mac-lane/cats-for-working-mathematician}
  Let $C$ be a category, $a, b$ be $C$-objects, and $f, g$ be
  $C$-morphisms from $a$ to $b$. The equalizer for the pair $(f,g)$
  is defined to be the limit (\ref{def/categorical-limit}) of the
  corresponding functor $P \xrightarrow{J} C$, where $P$ denotes
  the category with exactly two objects $0, 1$ and two
  non-identity morphisms $0 \rightrightarrows 1$.
\end{definition}

\begin{definition}[kernel]\label{def/kernel}\cite[VIII.1.]{mac-lane/cats-for-working-mathematician}
  Let $C$ be a category with a null object
  (\ref{def/null-object}). A kernel of a morphism
  $a \xrightarrow{f} b$ is defined to be an equalizer
  (\ref{def/equalizer}) for the pair $(f, a \xrightarrow{0} b)$,
  where $0$ denotes the zer morphism (\ref{def/zero-arrow}).
\end{definition}

\begin{definition}[cokernel]\label{def/cokernel}\cite[p.192]{mac-lane/cats-for-working-mathematician}
  The notion of a cokernel is the dual of the notion of a kernel
  (\ref{def/kernel}).
\end{definition}

\begin{definition}[pre-abelian
  category]\label{def/pre-abelian-category}
  An abelian category is an additive category
  (\ref{def/preadditive-category}) in which ever morphism has a
  kernel (\ref{def/kernel}) and a cokernel (\ref{def/cokernel}).
\end{definition}

\begin{lemma} \label{lemma/arrow-factors-thru-coim-and-im} Let
  $A$ be a pre-abelian category. Then each morphism $X \xrightarrow{f} Y$ in $A$
  has a canonical factorization
  \cite[I.1]{iversen/cohomology-of-sheaves}
  \begin{equation}
    f = \Big( X \to \cok(\ker(f)) \xrightarrow{f'} \ker(\cok(f)) \to Y \Big).
  \end{equation}
\end{lemma}

\begin{definition}[exact category] \label{def/exact-category} An
  exact category is a pre-abelian category in which the middle
  morphism of the canonical factorization
  (\ref{lemma/arrow-factors-thru-coim-and-im}) $f'$ of each morphism
  $X \xrightarrow{f} Y$ is an isomorphism.
\end{definition}

\begin{definition}[abelian category]\label{def/abelian-category}\cite[VIII.3. Definition]{mac-lane/cats-for-working-mathematician}
  An abelian category is an pre-abelian category
  (\ref{def/preadditive-category}) in which every monic morphism
  (\ref{def/monic-arrow}) is a kernel, and every epi morphism
  (\ref{def/epi-arrow}) is a cokernel.
\end{definition}

\begin{lemma}\label{lemma/exact-pre-abelian-is-abelian}
  Let $A$ be a pre-abelian category. Then the followings are
  equivalent.
  \begin{itemize}
    \item $A$ is abelian.
    \item $A$ is exact.
  \end{itemize}
\end{lemma}

\begin{lemma}\label{lemma/summary-for-abelian-category}
  \noindent To summarize, an abelian category is a category such
  that the followings are satisfied.
  \begin{itemize}
    \item (pre-additivity) Every hom set is an abelian group such
          that every composition is bilinear.
    \item (additivity) A null object and binary biproducts exist.
    \item (pre-abelianity) Every morphism has a kernel and cokernel.
    \item (exactness) Canonical factorizations induce isomophisms
          between the images and coimages
      $$\cok \circ \ker(-) \xrightarrow{~} \ker \circ \cok(-).$$
  \end{itemize}
\end{lemma}

\begin{lemma}
  Let $C$ be an abelian category and $D$ be an additive category.
  Suppose there is an additive equivalence of categories
  $C \xrightarrow{F} D$. Then $D$ is abelian.
\end{lemma}
\begin{Proof}
  By \ref{lemma/summary-for-abelian-category}, we have to show
  that the additive category $D$ is pre-abelian and exact.

  Let $X' \xrightarrow{\phi'} Y'$ be a $D$-morphism. Pick
  $C$-objects $X$ and $Y$ such that $FX$ and $FY$ are isomorphic
  to $X'$ and $Y'$ respectively. Since
  $$C(X,Y) \simeq D(FX,FY) \simeq D(X',Y'),$$
  $\phi'$ has kernels and cokernels, and its canonical
  factorization induces isomorphisms between images and coimages.
\end{Proof}

\subsection{Semisimple category}
Throughout the whole section, assume that $\mathbb{k}$ is an
algebraically closed field of characteristic $0$.

\begin{definition}[subobject]\cite[1.3.5]{egno/tensor-cats}
  Let $C$ be a category and $X$ be a $C$-object. A subobject of
  $X$ is a monic $C$-morphism $Y \xrightarrow{f} X$.
\end{definition}

\begin{definition}[simple object]\cite[1.5.1]{egno/tensor-cats}
  Let $C$ be an abelian category. A simple object $X$ of $C$ is a
  nonzero $C$-object whose only subobjects are
  $\mathbb{0} \xrightarrow{0} X$ and $X \xrightarrow{\id_{X}} X$.
\end{definition}

\begin{definition}[semisimple object]\cite[1.5.1]{egno/tensor-cats}
  Let $C$ be an abelian category. A semisimple object $X$ of $C$
  is a direct sum of some simple objects of $C$.
\end{definition}

\begin{definition}[semisimple category]
  \cite[1.5.1]{egno/tensor-cats} A semisimple category is an
  abelian category whose objects are all semisimple.
\end{definition}

\begin{definition}[object of finite length] \cite[1.5.3]{egno/tensor-cats}
  Let $X$ be an object of an abelian category $C$. We say that
  $X$ is of finite length if there exists a positive integer $n$
  and a sequence of monic morphisms
  \begin{equation}
    \mathbb{0} = X_{0} \to X_{1} \to \cdots \to X_{n} = X
  \end{equation}
  each of whose cokernel object $X_{i+1}/X_{i}$ is simple. Call
  $n$ the length of this sequence.
\end{definition}

\begin{remark} Such a sequence is called a Jordan-Holder series
  for $X$. By Jordan-Holder theorem
  \cite[1.5.4]{egno/tensor-cats}, all Jordan-Holder series of $X$
  have the same length.
\end{remark}

\begin{definition}[length of an object]\cite[1.5.5]{egno/tensor-cats}
  Let $C$ be an abelian category and $X$ a $C$-object. The length
  $X$ is defined to be the length of one, thus all, of its
  Jordan-Holder series.
\end{definition}

\begin{definition}[linear category over a field]
  \cite[1.2.2]{egno/tensor-cats}
  Let $\Bbbk$ be a field. A $\Bbbk$-linear category is an
  additive category $C$ whose hom-spaces are $\Bbbk$-vector
  spaces, such that all compositions of morphisms are
  $\Bbbk$-linear maps.
\end{definition}

\begin{definition}[locally finite abelian category over a
  field]\label{def/locally-finite-abelian-category-over-a-field}\cite[1.8.1]{egno/tensor-cats}
  A locally finite category (or an Artinian category) over
  $\Bbbk$ is a $\Bbbk$-linear abelian category $C$ that satisfies
  the following conditions.
  \begin{itemize}
    \item Every object has finite length.
    \item Every hom space is a finite dimensional $\Bbbk$-vector space.
  \end{itemize}
\end{definition}

\begin{definition}[finite abelian category over a field]
  \label{def/finite-abelian-category-over-a-field}\cite[1.8.6]{egno/tensor-cats}
  A finite category abelian category $C$ over $\Bbbk$ is a
  locally finite abelian category over $\Bbbk$ such that
  \begin{itemize}
    \item $C$ has enough projectives, i.e. every simple object of
          $C$ has a projective cover.
    \item The set of isomorphism classes of simple objects is
          finite.
  \end{itemize}
\end{definition}

\begin{remark}
  By discussion before \cite[1.8.6]{egno/tensor-cats}, a finite
  $\Bbbk$-linear abelian category $C$ is equivalent to the
  category of finite dimensional modules over a finite
  dimensional $\Bbbk$-algebra $A$.
\end{remark}

\subsection{Tensor category}\label{section/tensor-category}

Recall that the Crane-Yetter theory $CY$ comes in a family, of
which member depends on a type of algebraic data called the
premodular categories. Despite its technical definition (finite
semisimple ribbon braided rigid tensor category), it does not
hurt too much to think of a premodular category as a higher
categorical analogue of a finite abelian group: the ``braided
tensor'' structure encodes the (higher) group operation, the
``rigid'' structure encodes the (higher) inverses, and the
``ribbon'' structure ensures that $(g^{\m 1})^{\m 1}$ is
equivalent to $g$.

In this section, a complete definition for a premodular category
\ref{def/premodular-category} is collected from
\cite{egno/tensor-cats}. Throughout the whole section, assume
that $\mathbb{k}$ is an algebraically closed field of
characteristic $0$.

\begin{definition}[monoidal category]\cite[2.1.1]{egno/tensor-cats}
  A monoidal category is a septuple
  $$(C,\otimes,a,\mathbb{1},\iota,l,r),$$
  that satisfies the pentagon axiom and the triangle axiom
  \cite[(2.2)]{egno/tensor-cats}, where $C$ is a category,
  $C \times C \xrightarrow{\otimes} C$ is a bifunctor,
  $$((-_{1} \otimes -_{2}) \otimes -_{3}) \xrightarrow{a} (-_{1} \otimes (-_{2} \otimes -_{3}))$$
  is a natural equivalence, $\mathbb{1}$ is an object in $C$,
  $\mathbb{1} \otimes \mathbb{1} \xrightarrow{\iota} \mathbb{1}$
  is an isomorphism, $(-) \xrightarrow{l} (\mathbb{1} \otimes -)$
  and $(-) \xrightarrow{r} (- \otimes \mathbb{1})$ are natural
  equivalences.

  We will abuse notations and denote the septuple by $C$. The
  bifunctor $\otimes$ is called the tensor product bifunctor, the
  pair $(\mathbb{1},\iota)$ is called the unit object, and the
  natural equivalence $a$ is called the associativity isomorphism
\end{definition}

\begin{definition}[duals of an object]\cite[2.10.1 and
  2.10.2]{egno/tensor-cats}
  Let $X$ be an object of a monoidal category
  $(C,\otimes,\mathbb{1},a,\iota,l,r)$. A left dual of $X$ is an
  object $L$ with two morphisms
  \begin{align}
    L \otimes X     \xrightarrow{ev} \mathbb{1} \\
    \mathbb{1} \xrightarrow{coev} X \otimes L
  \end{align}
  such that the compositions of the following the identity
  morphisms
  \begin{align}
    X &\xrightarrow{coev \otimes 1} (X \otimes L) \otimes X \xrightarrow{a}  X \otimes (L \otimes X) \xrightarrow{1 \otimes ev} X, \\
    L &\xrightarrow{1 \otimes coev} L \otimes (X \otimes L) \xrightarrow{a^{(\m 1)}} (L \otimes X) \otimes L \xrightarrow{ev \otimes 1} L.
  \end{align}

  Similarly, a right dual of $X$ is an object $R$ with two
  morphisms
  \begin{align}
    X \otimes R     \xrightarrow{ev'} \mathbb{1} \\
    \mathbb{1} \xrightarrow{coev'} R \otimes X
  \end{align}
  such that the compositions of the following are the identity
  morphisms
  \begin{align}
    X &\xrightarrow{1 \otimes coev'} X \otimes (R \otimes X) \xrightarrow{a^{(\m 1)}} (X \otimes R) \otimes X \xrightarrow{ev' \otimes 1} X, \\
    R &\xrightarrow{coev' \otimes 1} (R \otimes X) \otimes R \xrightarrow{a}  R \otimes (X \otimes R) \xrightarrow{1 \otimes ev'} R.
  \end{align}
\end{definition}

\begin{remark} It can be proved that the left (resp., right)
  dual, if exists, is unique up to isomorphism
  \cite[2.10.5]{egno/tensor-cats}. We will denote it by
  $X^{\star}$ (resp., $^{\star}X$).
\end{remark}

\begin{definition}[rigid object]
  \cite[2.10.11]{egno/tensor-cats} Let $C$ be a monoidal
  category. A rigid object $X$ of $C$ is a $C$-object that has a
  left dual and a right dual.
\end{definition}

\begin{definition}[rigid category]\cite[2.10.11]{egno/tensor-cats}
  A rigid category $C$ is a monoidal category all of whose
  objects are rigid.
\end{definition}

\begin{definition}[multitensor category] \cite[4.1.1]{egno/tensor-cats}
  A multitensor category $C$ over $\Bbbk$ is a locally finite
  $\Bbbk$-linear abelian rigid monoidal category if the bifunctor
  $\otimes$ in the monoidal structure is $\Bbbk$-bilinear on
  morphisms.
\end{definition}

\begin{lemma}\label{lemma/tensor-is-biexact-in-multitensor-cats}\cite[4.2.1]{egno/tensor-cats}
  Let $C$ be a multitensor category. Then the bifunctor $\otimes$ is
  exact in both factors.
\end{lemma}
\begin{Proof}
  It is a fun exercise to prove. A sketch is as follows. Let $X$
  be a $C$-object. The rigidity says that $X \otimes (-)$ is a left and
  right adjoint functor. In general category theory, adjoint
  functors preserve all (co)limits essentially because $Hom$ does
  and Yoneda lemma. In particular, they preserve (co)kernels.
\end{Proof}

\begin{definition}[multifusion category]
  \cite[4.1.1]{egno/tensor-cats} A multifusion category over
  $\Bbbk$ is a multitensor category that is finite over $\Bbbk$
  and semisimple.
\end{definition}

\begin{definition}[fusion category]\cite[4.1.1]{egno/tensor-cats}\label{def/fusion-category}
  A fusion category $C$ is a multifusion category with
  $\End_{C}(\mathbb{1}) \simeq \Bbbk$.
\end{definition}

\begin{definition}[braiding]\cite[8.1.1]{egno/tensor-cats}
  A braiding of a monoidal category
  $(C,\otimes,\mathbb{1},a,\iota,l,r)$ is a natural equivalence
  \begin{equation}
    (-_{1} \otimes -_{2}) \xrightarrow{c} (-_{2} \otimes -_{1})
  \end{equation}
  such that the hexagon diagram \cite[(8.1)]{egno/tensor-cats}
  holds.
\end{definition}

\begin{definition}[braided category]
  \cite[8.1.2]{egno/tensor-cats}\label{def/braided-category} A braided category is a monoidal
  category with a braiding.
\end{definition}

\begin{remark}
  The Yang-Baxter equation holds automatically in a braided
  category \cite[8.1.10]{egno/tensor-cats}.
\end{remark}

\begin{definition}[half-braiding]
  \cite[(7.41)]{egno/tensor-cats}\label{def/half-braiding}
  A half-braiding for an object $X$ in a monoidal category
  $(C,\otimes,\mathbb{1},a,\iota,l,r)$ is a natural equivalence
  \begin{equation}
    (X \otimes -) \xrightarrow{c} (- \otimes X)
  \end{equation}
  such that the hexagon diagram \cite[(7.41)]{egno/tensor-cats}
  holds.
\end{definition}

\begin{definition}[twist]\cite[8.10.1]{egno/tensor-cats}
  Let $C$ be a braided rigid monoidal category. A twist of $C$ is
  an element $\theta \in \Aut(\id_{C})$ such that for each $C$-object
  $X, Y$
  \begin{equation}
    \theta_{{X \otimes Y}} = (\theta_{X} \otimes \theta_{Y}) \circ c_{Y,X} \circ c_{X,Y}
  \end{equation}
\end{definition}

\begin{definition}[ribbon structure]\cite[8.10.1]{egno/tensor-cats}
  Let $C$ be a braided rigid monoidal category. A twist $\theta$
  is called a ribbon structure if
  $(\theta_{X})^{\star} = (\theta_{(X^{\star})})$, where the
  first dual is taken in a rigid category.
\end{definition}

\begin{definition}[ribbon tensor category]\cite[8.10.1]{egno/tensor-cats}
  A ribbon tensor category is a braided rigid monoidal category
  equipped with a ribbon structure.
\end{definition}

\begin{definition}[premodular
  category]\label{def/premodular-category}\cite[8.13.1]{egno/tensor-cats}
  A premodular category is a ribbon fusion category.
\end{definition}

\begin{definition}[pivotal
  structure]\cite[4.7.7]{egno/tensor-cats}
  Let $C$ be a rigid monoidal category. A pivotal structure of
  $C$ is a natural isomorphism
  $$(-) \xrightarrow[\sim]{a} (-)^{\star\star}$$
  such that $a_{X \otimes Y} = a_{X} \otimes a_{Y}$ for all $C$-objects $X$
  and $Y$. We call a rigid monoidal category pivotal if it is
  equipped with a pivotal structure.
\end{definition}

\begin{definition}[pivotal dimension]
  Let $C$ be a rigid monoidal category with a pivotal structure
  $a$. Let $X$ be a $C$-object. We define the pivotal dimension
  with respect to $a$ to be
  $$\dim_{a}(X) := \mbox{Trace}(a_{X}) \in End_{C}(1).$$
\end{definition}

\begin{definition}[spherical
  structure]\cite[4.7.14]{egno/tensor-cats}
  Let $C$ be a rigid monoidal category with a pivotal structure
  $a$. The latter is called a spherical structure if
  $$\dim_{a}(X) = \dim_{a}(X^{\star})$$
  for any $C$-object $X$.
\end{definition}

\begin{remark}\cite[8.13.1]{egno/tensor-cats}
  Equivalently, a premodular category is also a braided fusion
  category equipped with a spherical structure.
\end{remark}

\begin{definition}[modular
  category]\label{def/modular-category}\cite[8.14 and 8.20.12]{egno/tensor-cats}
  A modular category is a premodular category with a
  non-degenerate S-matrix.
\end{definition}

\subsection{Adjunctions as monads} \label{section/adjunctions-as-monads}

\begin{definition}\label{def/adjunction}
  Let $X$ be a strict $2$-category, $C$ and $D$ be $X$-objects, and
  $C \xrightarrow{F} D$ and $C \xleftarrow{G} D$ be morphisms. We
  say that $F$ is right adjoint to $G$, that $G$ is
  left adjoint to $F$, if there exists $2$-morphisms

  $$1_{D} \xRightarrow{\eta} FG, \quad GF \xRightarrow{\epsilon} 1_{C}$$

  \noindent such that the followings

  $$G = G \circ 1_{D} \xRightarrow{1_{G} \ast \eta} G \circ F \circ G \xRightarrow{\epsilon \ast 1_{G}} G$$
  $$F = 1_{D} \circ F \xRightarrow{\eta \ast 1_{F}} F \circ G \circ F \xRightarrow{1_{F} \ast \epsilon} F$$

  \noindent equal the identity $2$-morphisms $1_{G}$ and $1_{F}$
  respectively. Denote $F \vdash G$ in this case.

  \noindent We call $\eta$ and $\epsilon$ the unit and counit of
  the monad, and call the coherence condition the rigidity
  condition.
\end{definition}

\begin{definition}\label{def/monad}
  Let $X$ be a strict $2$-category, $D$ be an $X$-object. Then
  $E=End_{X}(D)$ is a $1$-category. A monad of $D$ is a monoid
  object $T=(T,\eta,\mu)$ in $E$. That is to say, $T$ is an
  $E$-object, and $(1_{D} \xRightarrow{\eta} T)$ and
  $(T^{2} \xRightarrow{\mu} T)$ are $E$-morphisms such that

  $$(1_{T} \ast \eta) = 1_{T} = \mu \circ (\eta \ast 1_{T}), \qquad \mu \circ (\mu \ast 1_{T}) = \mu \circ (1_{T} \ast \mu).$$

\end{definition}

\begin{theorem}\label{thm/adjunctions-give-monads}
  Let $X$ be a strict $2$-category, $C$ and $D$ be $X$-objects,
  $(C \xrightarrow{F} D)$ and $(C \xleftarrow{G} D)$ be morphisms
  such that $F$ is right adjoint to $G$. Then $T=FG$ has a monad
  structure given by
  $$(1_{D} \xRightarrow{\eta} T), \quad (T^{2} \xRightarrow{\mu} T)$$
  where $\mu$ is defined as
  $FGFG \xRightarrow{1_{F} \ast \epsilon \ast 1_{G}} FG.$ Dually,
  $\perp = GF$ has a comonad structure.
\end{theorem}
\begin{Proof}
  \begin{align*}
      &\quad 1_{T} \\
    = &\quad 1_{F} \ast 1_{G} \\
    = &\quad 1_{F} \ast (G \xRightarrow{1_{G} \ast \eta} GFG \xRightarrow{\epsilon \ast 1_{G}} G) \\
    = &\quad (FG \circ 1_{D} \xRightarrow{1_{F} \ast 1_{G} \ast \eta} FGFG \xRightarrow{1_{F} \ast \epsilon \ast 1_{G}} FG) \\
    = &\quad (T \circ 1_{D} \xRightarrow{1_{T} \ast \eta} T \circ T \xRightarrow{\mu} T)
  \end{align*}

  \begin{align*}
      &\quad 1_{T} \\
    = &\quad 1_{F} \ast 1_{G} \\
    = &\quad (F \xRightarrow{\eta \ast 1_{F}} FGF \xRightarrow{1_{F} \ast \epsilon} F) \ast 1_{G} \\
    = &\quad (1_{D} \circ FG \xRightarrow{\eta \ast 1_{F} \ast 1_{G}} FGFG \xRightarrow{1_{F} \ast \epsilon \ast 1_{G}} FG) \\
    = &\quad (1_{D} \circ T \xRightarrow{\eta \ast 1_{T}} T \circ T \xRightarrow{\mu} T)
  \end{align*}

  \begin{align*}
      &\quad (T^{3} \xRightarrow{1_{T} \circ \mu} T^{2} \xRightarrow{\mu} T) \\
    = &\quad ((FG)(FGFG) \xRightarrow{1_{FG} \ast 1_{F} \ast \epsilon \ast 1_{G}} FGFG \xRightarrow{1_{F} \ast \epsilon \ast 1_{G}} FG) \\
    = &\quad 1_{F} \ast (GFGF \xRightarrow{1_{GF} \ast \epsilon} GF \xRightarrow{\epsilon} 1_{D}) \ast 1_{G} \\
    = &\quad 1_{F} \ast (GFGF \xRightarrow{\epsilon \ast \epsilon} 1_{D}) \ast 1_{G} \\
    = &\ldots = (T^{3} \xRightarrow{\mu \circ 1_{T}} T^{2} \xRightarrow{\mu} T).
  \end{align*}
\end{Proof}

\begin{theorem}
  Let $C$ and $D$ be categories, $C \xrightarrow{F} D$ and
  $C \xleftarrow{G} D$ be functors such that $F$ is a right
  adjoint functor to $G$, i.e. there exists natural equivalence
  $\Phi$ such that
  $$C(Gd,c) \xrightarrow[\Phi_{d,c}]{\sim} D(d,Fc).$$
  Then $F$ is right adjoint to $G$ the strict $2$-category $Cat$
  of all categories.
\end{theorem}
\begin{Proof}
  From the given natural equivalence $\Phi$, we have to construct
  $1_{D} \xRightarrow{\eta} FG$ and
  $GF \xRightarrow{\epsilon} 1_{C}$ that satisfy the conditions
  as in \ref{def/adjunction}. We contend that
  $\eta_{d} = \Phi_{d,Gd}(1_{Gd})$ and
  $\epsilon_{c} = \Phi^{(\m 1)}_{Fc,c}(1_{Fc})$ are as desired.

  Let us first show that $\eta_{d}$ is indeed a natural
  transformation from $1_{D}$ to $FG$. So is $\epsilon_{c}$
  similarly, whose proof will be omitted. Let
  $d \xrightarrow{\phi} d'$ be a $D$-morphism. We shall prove that
  the following commutative diagram commute.

  \begin{center}
    \begin{tikzcd}
      d \arrow[r, "\eta_{d}"]\arrow[d, "\phi"]
      & FGd \arrow[d, "FG\phi"] \\
      d' \arrow[r, "\eta_{d'}"] & D
    \end{tikzcd}
  \end{center}

  \noindent First notice that

  \begin{align*}
    \Phi^{(\m 1)}_{d,Gd'}(\eta_{d'} \circ \phi) = &\quad (G\phi)^{\star} \circ \Phi^{(\m 1)}_{d',Gd'} \circ (\phi^{\star})^{(\m 1)}(\eta_{d'} \circ \phi) \\
    = &\quad (G\phi)^{\star} \circ \Phi^{(\m 1)}_{d',Gd'}(\eta_{d'}) \\
    = &\quad (G\phi)^{\star}(1_{Gd'}) \\
    = &\quad G\phi
  \end{align*}

  \noindent As $\Phi$ is an equivalence, it suffices to prove that
  $\Phi^{(\m 1)}_{d,Gd'}((FG\phi) \circ \eta_{d})$ is also $G\phi$.

  \begin{align*}
    \Phi^{(\m 1)}_{d,Gd'}((FG\phi) \circ \eta_{d}) = &\quad \Phi^{(\m 1)}_{d,Gd'}((FG\phi) \circ \Phi_{d,Gd}(1_{Gd})) \\
    = &\quad (G\phi)_{\star} \circ \Phi^{(\m 1)}_{d,Gd} \circ (FG\phi)^{(\m 1)}_{\star} ((FG\phi) \circ \Phi_{d,Gd}(1_{Gd})) \\
    = &\quad (G\phi)_{\star} \circ \Phi^{(\m 1)}_{d,Gd} (\Phi_{d,Gd}(1_{Gd})) \\
    = &\quad (G\phi)_{\star}(1_{Gd}) \\
    = &\quad G\phi
  \end{align*}

  \noindent which is due to the naturality of $\Phi$

  \begin{center}
    \begin{tikzcd}
      C(Gd,Gd) \arrow[r, "\Phi_{d,Gd}"]\arrow[d,
      "(G\phi)_{\star}"]
      & D(d,FGd) \arrow[d, "(FG\phi)_{\star}"] \\
      C(Gd,Gd')' \arrow[r, "\Phi_{d,Gd'}"] & D(d,FGd')
    \end{tikzcd}
  \end{center}

  Next, we have to show the rigidity conditions of $\eta$ and $\epsilon$.
  Indeed, from the naturality of $\Phi$ we have the commutative
  diagram

  \begin{center}
    \begin{tikzcd}
      C(GFGd,Gd) \arrow[r, "\Phi_{FGd,Gd}"]\arrow[d,
      "(G(\eta_{d})^{\star})"]
      & D(FGd,FGd) \arrow[d, "\eta_d^{\star}"] \\
      C(Gd,Gd) \arrow[r, "\Phi_{d,Gd}"] & D(d,FGd)
    \end{tikzcd}
  \end{center}

  \noindent and thus the following holds

  \begin{align*}
    [(\epsilon \ast 1_{G}) \circ (1_{G} \ast \eta)]_{d} = &\quad (\epsilon \ast 1_{G})_{d} \circ (1_{G} \ast \eta)_{d} \\
    = &\quad (GFGd \xrightarrow{\epsilon_{Gd}} Gd) \circ G(d \xrightarrow{\eta_{d}} FGd) \\
    = &\quad Gd \xrightarrow{G(\eta_{d})} GFGd \xrightarrow{\epsilon_{Gd}} Gd \\
    = &\quad Gd \xrightarrow{G(\eta_{d})} GFGd \xrightarrow{\Phi^{(\m 1)}_{FGd,Gd}(1_{(FGd)})} Gd \\
    = &\quad (G(\eta_{d}))^{\star}\Big(GFGd \xrightarrow{\Phi^{(\m 1)}_{FGd,Gd}(1_{(FGd)})} Gd \Big) \\
    = &\quad \big(\Phi^{(\m 1)}_{d,Gd} \circ \eta^{\star}_{d}\big)(1_{FGd}) \\
    = &\quad \Phi^{(\m 1)}_{d,Gd}(\eta_{d}) = 1_{Gd}.
  \end{align*}

  \noindent The other rigidity condition follows similarly from the
  commutative diagram

  \begin{center}
    \begin{tikzcd}
      C(GFc,GFc) \arrow[r, "\Phi_{Fc,GFc}"]\arrow[d,
      "(\epsilon_{c})_{\star}"]
      & D(Fc,FGFc) \arrow[d, "(F\epsilon_{c})_{\star}"] \\
      C(GFc,c) \arrow[r, "\Phi_{Fc,c}"] & D(Fc,Fc)
    \end{tikzcd}
  \end{center}

\end{Proof}

Therefore, adjoint functors give adjoint pairs in the
$2$-category $Cat$, which in turn gives a monad and a comonad in
$Cat$. Let's summarize the result in the following theorem.

\begin{theorem}
  Let $C \xrightarrow{F} D$ and $C \xleftarrow{G} D$ be functors
  such that $F$ is right adjoint to $G$. Then there is a
  $D$-monad $(T=FG, \eta, \mu)$ and a $C$-comonad $(\perp=GF, \epsilon, \Delta)$, where
  $$\eta_{d} = \Phi_{d,Gd}(1_{(Gd)}),
  \mu = 1_{F} \ast \epsilon \ast 1_{G},$$
  $$\epsilon_{c} = \Phi^{(\m 1)}_{Fc,c}(1_{Fc}),
  \Delta = 1_{G} \ast \eta \ast 1_{F}.$$
\end{theorem}

For the rest of this subsection, assume that
$C \xrightarrow{F} D$ is a right adjoint functor of
$C \xleftarrow{G} D$ with the natural transformation $\Phi$. The
categories $C$ and $D$ are intimately tided together by the
adjoint functors between them. For example, a part of
compositions in $C$ can be identified as monadic composition in
$D$.

\begin{theorem}
  The usual composition map
  $$C(Gx,Gy) \times C(Gy,Gz) \xrightarrow{\circ} C(Gx,Gz)$$
  is identified under $\Phi$ as the Kleisi composition
  $$D(x,Ty) \times D(y,Tz) \xrightarrow{\circ_{T}} D(x,Tz)$$
  $$ \circ_{T}(f, g) \mapsto \mu_{Z} \circ (Tg) \circ f. $$
\end{theorem}
\begin{Proof}
  Since $\Phi$ is an equivalence, and since
  $$\Phi(\circ (\Phi^{(\m 1)}(f,g))) = \Phi_{x,Gz}(\Phi^{(\m 1)}_{y,Gz}(g) \circ \Phi^{(\m 1)}_{x,Gy}(f)),$$
  it suffices to prove that
  $$\Phi^{(\m 1)}_{x,Gz}(\mu_{Z} \circ (Tg) \circ f) = \Phi_{y,Gz}^{(\m 1)}(g) \circ \Phi_{x,Gy}^{(\m 1)}(f).$$
  The main task would be to express $\mu_{Z}$ in terms of $\Phi$.

  From the commutative diagram

  \begin{center}
    \begin{tikzcd}
      C(GFGz,Gz) \arrow[r, "\Phi_{FGz,Gz}"] \arrow[d, "(-) \circ
      G(g)"]
      & D(FGz,FGz) \arrow[d, "(-) \circ g"] \\
      C(Gy,Gz) \arrow[r, "\Phi_{y,Gz}"] & D(y,FGz)
    \end{tikzcd}
  \end{center}

  \noindent we have

  \begin{align*}
    &\quad \mu_{Z} \circ (Tg) \circ f \\
    = &\quad F(\Phi^{(\m 1)}_{FGz,Gz}(1_{FGz}) \circ F(G(g)) \circ f) \\
    = &\quad F(\Phi^{(\m 1)}_{FGz,Gz}(1_{FGz}) \circ G(g)) \circ f) \\
    = &\quad F(\Phi^{(\m 1)}_{y,Gz}(1_{FGz} \circ g)) \circ f) \\
    = &\quad F(\Phi^{(\m 1)}_{y,Gz}(g)) \circ f)
  \end{align*}

  \noindent It remains to prove that
  $${\Phi^{(\m 1)}_{x,Gz}}(F(\Phi^{(\m 1)}_{y,Gz}) \circ f) = \Phi^{(\m 1)}_{y,Gz}(g) \circ \Phi^{(\m 1)}_{x,Gy}(f),$$
  which directly follows from the commutative diagram obtained
  from the naturality of $\Phi$:

  \begin{center}
    \begin{tikzcd}
      C(Gx,Gy) \arrow[r, "\Phi_{x,Gy}"]\arrow[d, "\Phi^{(\m 1)}(g) \circ
      (-)"]
      & D(x,FGy) \arrow[d, "F(\Phi^{(\m 1)}) \circ (-)"] \\
      C(Gx,Gz) \arrow[r, "\Phi_{x,Gz}"] & D(x,FGz)
    \end{tikzcd}
  \end{center}

\end{Proof}

The natural equivalence $\Phi$ isn't as easy to manipulate as the
(co)monads it induces. We collect more statements that express
the former in terms of the later.

\begin{lemma}\label{lemma/adjunction-in-monadic-term}
  In terms of (co)monad, $\Phi$ can be expressed as follows.
  $$\Phi_{d,c}(\phi) = F(\phi) \circ \eta_{d},$$
  $$\Phi^{(\m 1)}_{d,c}(\psi) = \epsilon_{c} \circ G(\psi).$$
\end{lemma}
\begin{Proof}
  These are evident from the commutative diagrams below
  respectively.

  \begin{center}
    \begin{tikzcd}
      C(Gd,Gd) \arrow[r, "\Phi_{d,Gd}"]\arrow[d, "\Phi \circ
      (-)"]
      & D(d,FGd) \arrow[d, "(F\phi) \circ (-)"] \\
      C(Gd,c) \arrow[r, "\Phi_{d,c}"] & D(d,Fc)
    \end{tikzcd}
    \begin{tikzcd}
      C(GFc,Fc) \arrow[r, "\Phi_{Fc,c}"]\arrow[d, "(-) \circ
      \psi"]
      & D(Fc,Fc) \arrow[d, "(-) \circ G(\psi)"] \\
      C(Gd,c) \arrow[r, "\Phi_{d,Fc}"] & D(d,Fc)
    \end{tikzcd}
  \end{center}

\end{Proof}

Instances arise where two functors are adjoint to each other from
both sides. We call them (strict) ambidextrous functors.

\begin{definition}[ambidextrous adjunctions]
  Let $C \xrightarrow{F} D$ and ${C \xleftarrow{G} D}$ be
  functors. We call $F$ and $G$ a pair of (strict) ambidextrous
  functors if $F$ is both left-adjoint and right-adjoint to $G$.
\end{definition}

Two monads and two comonads arise from a pair of ambidextrous
functors. More precisely, that $F \vdash G$ gives a natural $D$-monad
$(T=FG, \eta, \mu)$ and a natural $C$-comonad $(\perp=GF, \epsilon, \Delta)$.
Similarly, that $F \dashv G$ gives a natural $D$-comonad
$(T=FG, \eta', \mu')$ and a natural $C$-monad $(\perp=GF, \epsilon',\Delta')$. In
particular, we have a bimonad structure on $T=(T,\eta,\mu,\epsilon,\Delta)$, with
unit $\eta$, multiplication $\mu$, counit $\epsilon$, and co-multiplication
$\Delta$.

\begin{definition}\label{def/unity-trace}
  We the bimonad $T$ is of unity trace if
  $$(1_{D} \xrightarrow{\eta} T \xrightarrow{\eta'} 1_{D}) = (1_{D} \xrightarrow{1_{(1_{D})}} 1_{D}).$$
  We the bimonad $T$ is of collapsable diamond if
  $$(T \xrightarrow{\mu'} T^{2} \xrightarrow{\mu} T) = (T \xrightarrow{1_{T}} T).$$
\end{definition}

\noindent By the following lemma, the second condition is
superseded by the first one.

\begin{lemma}
  If such adjunction is of unity trace for $\perp$, then $T$ is of
  collapsable diamond.
\end{lemma}
\begin{Proof}
  \begin{align*}
    &\quad (T \xrightarrow{\mu'} T^{2} \xrightarrow{\mu} T) \\
    = &\quad (FG \xrightarrow{1_{F} \ast \epsilon' \ast 1_{G}} F(GF)G \xrightarrow{1_{F} \ast \epsilon \ast 1_{G}} FG) \\
    = &\quad 1_{F} \ast (1 \xrightarrow{\epsilon'} \perp \xrightarrow{\epsilon} 1) \ast 1_{G} \\
    = &\quad 1_{F} \ast 1_{(1_{C})} \ast 1_{G} = 1_{T}
  \end{align*}
\end{Proof}

The unity trace condition turns out to be crucial for our work -
essentially it guarantees an averaging map analogue to that in
the theory of finite group representations.

\begin{lemma}
  Let $T = (T, \eta, \mu, \epsilon, \Delta)$ be a $D$-bimonad of unity trace. Then
  for each $D$-object $x$ and $y$, the morphism
  $[D(x,y) \xrightarrow{\eta_{y} \circ (-)} D(x,Ty)]$ is monic, the
  arrow $[D(x,Ty) \xrightarrow{\eta'_{y} \circ (-)} D(x,y)]$ is epic,
  and moreover the map $(\eta_{y} \circ \eta_{y}')$ is a projection map onto
  the image of $(\eta_{y} \circ -)$.
\end{lemma}
\begin{Proof}
  The unity trace condition says that $\eta_{y}' \circ \eta_{y} = 1_{y}$,
  so the first two conditions follow. The last statement is also
  evident since
  $$(\eta'_{y}\eta_{y})^{2} = \eta'_{y} 1_{y} \eta_{y} = \eta'_{y}\eta_{y}.$$
\end{Proof}

Therefore, in the context of (strict) ambidextrous adjoint
functors, the unity trace condition yields a projection map
$$C(Gx,Gy) \xrightarrow{\pi_{x,y}} D(x,y)$$ from the equivalence
$C(Gx,Gy) \simeq D(x,Ty)$. In the next lemma, we see that this
projection is functorial without any extra assumption.

\begin{theorem}\label{thm/unity-trace-implies-monadic-projection}
  Let $F : C \leftrightarrow D : G$ be a pair of strictly ambidextrous adjoint
  functors, and $T=FG$ be the naturally induced bimonad on $D$.
  If $T$ is of unity trace, then $\pi_{x,y}$ is functorial in the
  sense that
  \begin{enumerate}
    \item $\pi_{x,x}(Gx \xrightarrow{1_{Gx}} Gx) = 1_{x}.$
    \item For $C$-morphisms
          $(Gx \xrightarrow{\phi} Gy \xrightarrow{\sigma} Gz)$, we
          have $$(x \xrightarrow{\pi_{x,y}\phi} y \xrightarrow{\pi_{y,z}\sigma} z) = (x \xrightarrow{\pi_{x,z}(\sigma\phi)} z).$$
  \end{enumerate}
\end{theorem}
\begin{Proof}
  By the unity trace condition and
  \ref{lemma/adjunction-in-monadic-term},
  $$\pi_{x,x}(1_{Gx}) = \eta'_{x} \circ (1_{FGx} \circ \eta_{x}) = 1_{X},$$
  proving the first statement. It remains to prove that
  $$ (\eta_{z}' \circ F(\sigma) \circ \eta_{y}) \circ (\eta_{y}' \circ F(\phi) \circ \eta_{x}) = (\eta_{z}' \circ F(\sigma\phi) \circ \eta_{x}).$$
  Indeed,

  \begin{align*}
    &\quad \Big( z \xleftarrow{\eta'_{z}} Tz \xleftarrow{F(\sigma)} Ty \xleftarrow{\eta_{y}} y \xleftarrow{\eta'_{y}} Ty \xleftarrow{F(\phi)} Tx \xleftarrow{\eta_{x}} x \Big) \\
    = &\quad \Big( z \xleftarrow{\eta'_{z}} Tz \xleftarrow{F(\sigma)} Ty \xleftarrow{T\eta_{y}'} T^{2}y \xleftarrow{\eta_{Ty}} Ty \xleftarrow{F(\phi)} Tx \xleftarrow{\eta_{x}} x \Big)\\
    = &\quad \Big( z \xleftarrow{\eta'_{z}} Tz \xleftarrow{F(\sigma)} Ty \xleftarrow{T\eta_{y}'} T^{2}y \xleftarrow{T(F(\phi))} T^{2}x \xleftarrow{\eta_{Tx}} Tx \xleftarrow{\eta_{x}} x \Big)\\
    = &\quad \Big( z \xleftarrow{\eta'_{z}} Tz \xleftarrow{F(\sigma)} Ty \xleftarrow{T\eta_{y}'} T^{2}y \xleftarrow{T(F(\phi))} T^{2}x \xleftarrow{T\eta_{x}} Tx \xleftarrow{\eta_{x}} x \Big)\\
    = &\quad \Big( z \xleftarrow{\eta'_{z}} Tz \xleftarrow{F(\sigma)} Ty \xleftarrow{F(\phi)} Tx \xleftarrow{\eta_{x}} x \Big)
  \end{align*}

  \noindent The first two equalities follow from the naturality
  of $\eta$. The third equality $T_{(\eta_{x})} = \eta_{(Tx)}$ follows from
  the Eckmann-Hilton argument
  $$\eta \ast 1_{T} = 1_{T} \ast \eta.$$
  Finally, the last equality follows from all the tricks and
  conditions: that $\eta$ is natural, that $T \circ \eta = \eta \circ T$, that $T$
  is functorial, and the unity trace condition.
\end{Proof}

\begin{remark}
  In the context of categorical center of higher genera, the
  proof above translates into the following graphical proof,
  where the orange dotted lines represent the shorthand notation
  $\Omega$ given in \ref{def/Omega}.
  \begin{center}
    \includegraphics[height=4cm]{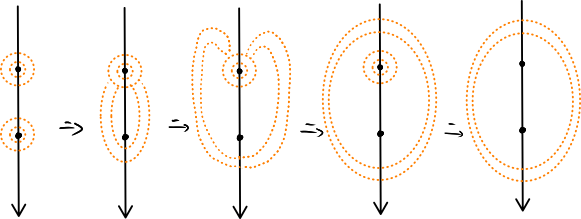}
  \end{center}
\end{remark}

In the proof, it is tempting to demand $\eta_{y} \circ \eta_{y}'$ to be
identity, which would have finished the proof right away.
However, it is not necessarily true. In fact, it is false in our
context. The best one can say about it is that it is idempotent.

\begin{example}\label{example/main-example-of-monadic-projection}
  Let $C$ be a premodular category, $\sigma$ be an admissible gluing,
  $D$ be the categorical center of higher genera $Z_{\sigma}(C)$, $F$
  be the induction functor $C \xrightarrow{I_{\sigma}} Z_{\sigma}(C)$, and
  $G$ be the forgetful functor $C \xleftarrow{F_{\sigma}} Z_{\sigma}(C)$.

  By \ref{theorem/F-and-I-are-ambidextrous-adjoint}, both
  functors are strictly ambidextrous to each other. Moreover, it
  is clear by their definitions that the bi-monads they form are
  of unity trace. Therefore, by
  \ref{thm/unity-trace-implies-monadic-projection}, we have the
  followings.

  \begin{enumerate}
    \item $D((X,\gamma), (Y, \beta))$ embeds into $C(X,Y)$ naturally, with
          a natural projection $\pi_{\gamma,\beta}$ onto the subspace.
    \item $C(X,Y)$ embeds into $D(I_{\sigma}(X), I_{\sigma}(Y))$ naturally,
          with a natural projection $\pi_{X,Y}$ onto the subspace.
  \end{enumerate}
\end{example}

This is an analogue of the averaging map one has in the theory of
finite dimensional complex linear representations of finite
groups.

\subsection{Misc proofs}

Statements and proofs that could break the flow of are collected
in this section. The readers are advised to use it as a
reference.

\begin{lemma}\label{lemma/higher-cat-center-is-abelian}
  Let $C$ be a premodular category and $\sigma \in \Adm_{2n}$ an
  admissible gluing. Then the categorical center of higher genera
  $Z_{\sigma}(C)$ is an abelian category.
\end{lemma}
\begin{Proof}
  By \ref{lemma/summary-for-abelian-category} we need to show
  that $Z_{\sigma}(C)$ is an additive category such that

  \begin{itemize}
    \item every morphism in $Z_{\sigma}(C)$ has a kernel and a cokernel.
    \item $Z_{\sigma}(C)$ is an exact category.
  \end{itemize}

  \noindent By its definition, $Z_{\sigma}(C)$ is additive. To prove
  that every morphism has a kernel and a cokernel, first let
  $(X,\gamma) \xrightarrow{f} (Y,\beta)$ be a $Z_{\sigma}(C)$-morphism. Recall by
  definition that $f$ is a $C$-morphism $X \xrightarrow{f} Y$ that
  respects both sets of half-braidings $\gamma$ and $\beta$. Since $C$ is
  premodular thus abelian, $f$ has a kernel
  $K \xhookrightarrow{m} X$ in $C$. We will construct a
  $Z_{\sigma}(C)$-object $(K, m^{\star}\gamma)$ such that
  $(K, m^{\star}) \xrightarrow{m} (X,\gamma)$ is a $Z_{\sigma}(C)$-morphism and is
  in fact a kernel of $f$.

  To construct $m^{\star}\gamma$, notice that all we need is a set of
  half-braidings for $K$ that work compatibly with $\gamma$. As
  $$K \xrightarrow{m} X \xrightarrow{f} Y$$
  is exact and that $\otimes$ is bi-exact
  \ref{lemma/tensor-is-biexact-in-multitensor-cats}, we see that
  $1_{(-)} \otimes m$ and $m \otimes 1_{(-)}$ are kernels of $1_{(-)} \otimes f$
  and $f \otimes 1_{(-)}$ respectively. Therefore,
  $\gamma_{[i]} \circ (m \otimes 1_{(-)})$ factors through $1_{(-)} \otimes m$
  uniquely. Ditto for the other direction. So defines an natural
  isomorphism $$ K \otimes (-) \xrightarrow{(m^{\star}\gamma)_{[i]}} (-) \otimes K.$$
  Define so for all other $i$'s. it's straightforward to prove
  that $m^{\star}\gamma$ is a $\sigma$-pair \ref{def/sigma-pair} from that $\gamma$
  is also one.

  From the construction above, clearly
  $$(K, m^{\star}\gamma) \xrightarrow{m} (X, \gamma)$$
  is a $Z_{\sigma}(C)$-arrow. It remains to show that $m$ is
  indeed a kernel of $f$ in $Z_{\sigma}(C)$. Let
  $(W,\alpha) \xrightarrow{h} (X,\gamma)$ be a
  $Z_{\sigma}(C)$-morphism such that $fh = 0$. Then $h$ uniquely
  factors through $m$ by some $C$-morphism $k$. The crux is to
  show that $k$ is indeed a $Z_{\sigma}(C)$-morphism. But indeed,
  by the projection
  \ref{example/main-example-of-monadic-projection} we have
  \begin{align*}
    &\quad h = m \circ k \\
    \Rightarrow &\quad \pi(h) = \pi(m \circ k) = \pi(m) \circ \pi(k) \\
    \Rightarrow &\quad h = m \circ \pi(k)
  \end{align*}
  But since $k$ is unique, we have $k = \pi(k)$, which is indeed a
  morphism in $Z_{\sigma}(C)$. Therefore, $m$ is a kernel of $f$. The
  argument works for the cokernel, and is thus omitted. A
  corollary of this construction is that the (co)kernels are
  really the same as in $C$, so $Z_{\sigma}(C)$ is clearly exact since
  $C$ is exact.
\end{Proof}

\begin{lemma}\label{lemma/higher-cat-center-is-semisimple}
  Let $C$ be a premodular category and $\sigma \in \Adm_{2n}$ an
  admissible gluing. Recall from
  \ref{lemma/higher-cat-center-is-abelian} that the categorical
  center of higher genera $Z_{\sigma}(C)$ is abelian. Moreover, it is
  semisimple.
\end{lemma}
\begin{Proof}
  Let $(X,\gamma) \xrightarrow{f} (Y,\beta)$ be a monic morphism in
  $Z_{\sigma}(C)$. It suffices to show that $f$ has a left inverse.
  Recall that $f$ is also a $C$-morphism $X \xrightarrow{f} Y$. We
  contend that $X \xrightarrow{f} Y$ is monic in $C$. Indeed,
  assume
  $$(W \xrightarrow{g} X \xrightarrow{f} Y) = (W \xrightarrow{h} X \xrightarrow{f} Y)$$
  then
  $$
  (I_{\sigma}W \xrightarrow{\overline{g}} (X,\gamma) \xrightarrow{f} (Y,\beta)) =(I_{\sigma}W \xrightarrow{\overline{h}} (X,\gamma) \xrightarrow{f} (Y,\beta))$$
  by the construction of $I_{\sigma}$. Then
  $\overline{g} = \overline{h}$, and thus $g=h$.

  Since $C$ is semisimple, we get a left inverse
  $X \xleftarrow{p} Y$ for free. However, $p$ lives in $C$, so we
  need to find another candidate that does the job in $Z_{\sigma}(C)$.
  This is again taken care by the projection
  \ref{example/main-example-of-monadic-projection} We contend
  that it is a left inverse of $f$ in $Z_{\sigma}(C)$. Indeed, as
  $\pi_{\beta,\gamma}$ is a projection, $\pi_{\beta,\gamma}(f) = f$. So,
  $$
  \pi_{\beta,\gamma}(p) \circ f = \pi_{\beta,\gamma}(p) \circ \pi_{\gamma,\beta}(f) = \pi_{\gamma,\gamma}(p \circ f) = \pi_{\gamma,\gamma}(1_{X}) = 1_{X}.
  $$
\end{Proof}

\begin{lemma}\label{lemma/higher-cat-center-is-finite}
  Let $C$ be a premodular category and $\sigma \in \Adm_{2n}$ an
  admissible gluing. Recall from
  \ref{lemma/higher-cat-center-is-abelian} and
  \ref{lemma/higher-cat-center-is-semisimple} that the
  categorical center of higher genera $Z_{\sigma}(C)$ is semisimple
  abelian. Moreover, it is finite.
\end{lemma}
\begin{Proof}
  To prove that $Z_{\sigma}(C)$ is finite, we turn to the finiteness
  of $C$. Since $Z_{\sigma}(C)$ is a $\mathbb{k}$-linear abelian
  category by construction, from
  \ref{def/locally-finite-abelian-category-over-a-field} and
  \ref{def/finite-abelian-category-over-a-field} we only have to
  show four things.

  \begin{itemize}
    \item Every object has finite length.
    \item Every hom space is a finite dimensional
          $\mathbb{k}$-vector space.
    \item $Z_{\sigma}(C)$ has enough projectives, i.e. every simple
          object of $Z_{\sigma}(C)$ has a projective cover.
    \item The set of isomorphism classes of simple objects is
          finite.
  \end{itemize}

  To prove that every object has finite length, pass a simple
  filtration of an object in $Z_{\sigma}(C)$ to one in $C$ by the
  forgetful functor $F_{\sigma}$. Extend the latter to a simple
  filtration in $C$, which has finite length as $C$ is assumed
  finite. Thus the former is also of finite length. To prove that
  every hom space is of finite dimensional, recall that the
  morphism spaces of $Z_{\sigma}(C)$ are defined as subspaces of those
  of $C$. Therefore the dimension of the former is bounded by the
  dimension of the later, which is finite the finiteness
  assumption of $C$.

  To prove that $Z_{\sigma}(C)$ has enough projectives, it
  suffices to show that $Z_{\sigma}(C)$ is semisimple, as then
  each epic morphism admits a left inverse. But this fact has been
  shown in \ref{lemma/higher-cat-center-is-semisimple}. To prove
  that there are only finitely many simple objects (up to
  isomorphism), we utilize the ambidextrous adjunction of
  $F_{\sigma}$ and $I_{\sigma}$. Let $(X,\gamma)$ be a simple
  object of $Z_{\sigma}(C)$. From
  $$Hom_{C}(X,Y) \simeq Hom_{Z_{\sigma}(C)}((X,\gamma), I_{\sigma}(Y))$$
  we know that $(X,\gamma)$ appears as a summand in $I(Y)$ for any $Y$
  that appears as a summand in $X$. Since $C$ has finitely many
  simple objects (up to isomorphism), it follows that there are
  finitely many such $(X,\gamma)$.
\end{Proof}

\begin{lemma}\label{lemma/induced-half-braiding-satisfies-comm-relations}
  The set of half-braidings defined in
  \ref{def/induced-half-braiding} satisfies the pairwise
  commutative relations \ref{def/comm-relation}.
\end{lemma}
\begin{Proof}
  By definition, we have to prove that for each $i, j$, $\gamma_{i}$
  and $\gamma_{j}$ satisfies the commutative relation posed in
  \ref{def/comm-relation}. In this pictorial proof, we use the
  color light-blue to indicate $\gamma_{i}$ and the color red to
  indicate $\gamma_{j}$. Recall that with out loss of generality,
  there are three cases to consider

  $$ 1. \quad [i]' < [i]'' < [j]' < [j]''$$
  \begin{center}
    \includegraphics[height=10cm]{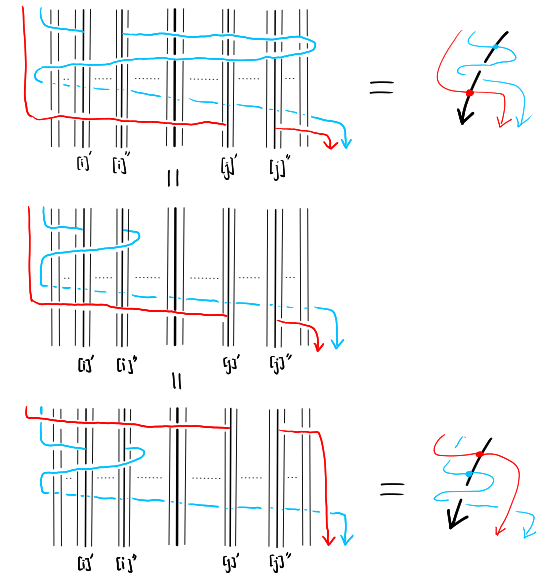}
  \end{center}

  $$ 2. [i]' < [j]' < [i]'' < [j]''$$
  \begin{center}
    \includegraphics[height=7cm]{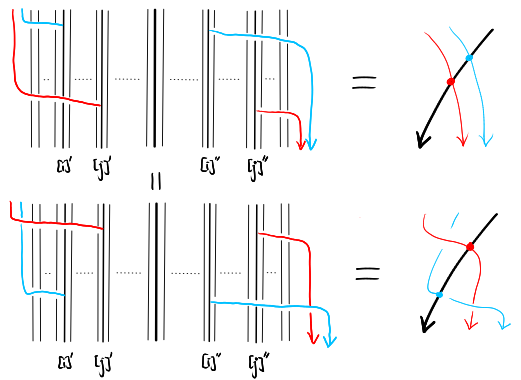}
  \end{center}

  $$3. [i]' < [j]' < [j]'' < [i]''$$
  \begin{center}
    \includegraphics[height=7cm]{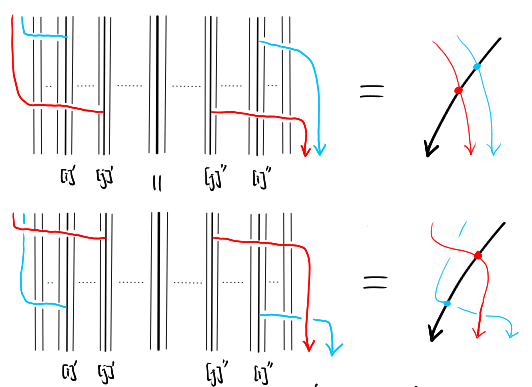}
  \end{center}

\end{Proof}

\begin{lemma}\label{lemma/induced-arrow-respect-half-braidings}
  The induced morphisms in (\ref{def/induced-arrow}) is compatible
  with the sets of half-braidings $\gamma$ and $\beta$ given in
  (\ref{def/induced-half-braiding}).
\end{lemma}
\begin{Proof}
  Clearly it holds from the following figure.
  \begin{center}
    \includegraphics[height=10cm]{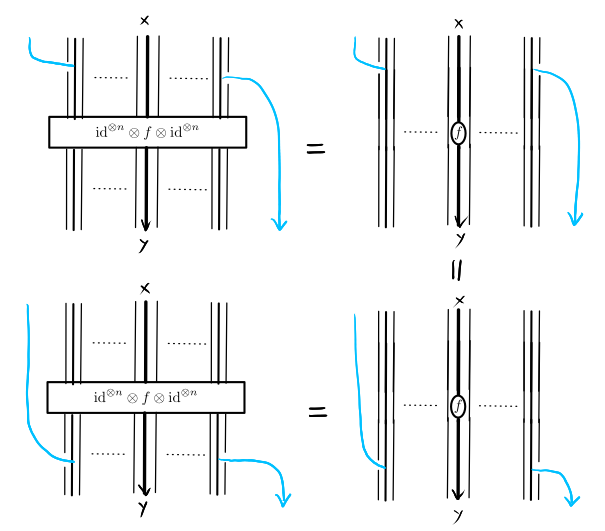}
  \end{center}
\end{Proof}

\begin{lemma}\label{lemma/additive-functor-to-abelian-cat-lifts-to-karoubi}
  Let $A$ be an additive category and $B$ be an abelian category.
  Suppose $$A \xrightarrow{\phi} B$$ is an additive functor. Then
  $\phi$ lifts additively to the Karoubi completion $\Kar(A)$ of
  $A$:
  $$\Kar(A) \xrightarrow{\Phi} B.$$
\end{lemma}
\begin{Proof}

  Given the assumptions, we must construct $\Phi$ explicitly. Recall
  that a typical object of $\Kar(A)$ is $\bar{X}:=(X,p)$ of
  $X \in \Obj(A)$ and an idempotent $p \in End_{A}(X)$. Define
  $\Phi(\bar{X})$ to be $im_{B}(\phi(p))$. Recall also that a typical
  morphism $$(X,p) \xrightarrow{f} (Y,q) $$ is an $A$-morphism
  $X \xrightarrow{f} Y$ such that $f = qfp$. Hence $\Phi(f)$ induces
  a $B$-morphism $$\im(\phi(p)) \xrightarrow{\phi(f))} \im(\phi(q)).$$ Define
  it to be $\Phi(f)$. So defined map $\Phi$ is clearly an additive
  functor that extends $\phi$.
\end{Proof}

\printbibliography
\end{document}